\documentclass{amsart}
\usepackage{ucs}
\usepackage[utf8x]{inputenc}
\usepackage{verbatim}
\usepackage{paralist}
\usepackage{leftidx}
\usepackage{dsfont}
\usepackage{amsthm}
\usepackage{amsmath}
\usepackage{amssymb}
\usepackage{amscd}
\usepackage{graphicx}
\usepackage{setspace}
\usepackage[all]{xy}
\usepackage{mathrsfs} 
\usepackage{hyperref}
\title{On arithmetic intersection numbers on self-products of curves}
\author{Robert Wilms}
\address{Robert Wilms\\
	Department of Mathematics and Computer Science\\
	University of Basel\\
	Spiegelgasse 1\\
	4051 Basel\\
	Switzerland}
\email{robert.wilms@unibas.ch}
\thanks{The author gratefully acknowledges support from SFB/Transregio 45.}
\begin{document}
	\numberwithin{equation}{section}
	\newtheorem{Def}{Definition}
	\numberwithin{Def}{section}
	\newtheorem{Rem}[Def]{Remark}
	\newtheorem{Lem}[Def]{Lemma}
	\newtheorem{Que}[Def]{Question}
	\newtheorem{Cor}[Def]{Corollary}
	\newtheorem{Exam}[Def]{Example}
	\newtheorem{Thm}[Def]{Theorem}
	\newtheorem*{clm}{Claim}
	\newtheorem{Pro}[Def]{Proposition}
	\newcommand\gf[2]{\genfrac{}{}{0pt}{}{#1}{#2}}

\begin{abstract}
We give a closed formula for the N\'eron--Tate height of tautological integral cycles on Jacobians of curves over number fields as well as a new lower bound for the arithmetic self-intersection number $\hat{\omega}^2$ of the dualizing sheaf of a curve in terms of Zhang's invariant $\varphi$. As an application, we obtain an effective Bogomolov-type result for the tautological cycles. We deduce these results from a more general combinatorial computation of arithmetic intersection numbers of adelic line bundles on higher self-products of curves, which are linear combinations of pullbacks of line bundles on the curve and the diagonal bundle. 
\end{abstract}
\maketitle
\section{Introduction}
In \cite{Wil17} we introduced a combinatorial method to compute Deligne pairings on self-products of universal families of Riemann surfaces. De Jong \cite{dJo18} applied this technique to compute N\'eron--Tate heights of certain tautological cycles on Jacobians of curves over number fields. The aim of this paper is to establish an improvement of this method, which makes it possible to deduce a closed formula for the N\'eron--Tate height of the tautological cycles and in particular, to obtain an effective Bogomolov-type result for these cycles.

To give more precise statements, let $X$ be a smooth projective geometrically connected curve of genus $g\ge 2$ over a number field $K$ of degree $d_K=[K:\mathbb{Q}]$ with semistable reduction over the integers $\mathcal{O}_K$ of $K$. Further, let $J=\mathrm{Pic}^0(X)$ be the Jacobian of $X$ and $\mathcal{L}$ an ample symmetric line bundle on $J$, which induces the canonical principal polarization of $J$ and which is rigidified at the origin.

For any divisor $\alpha$ on $X$ of degree $1$ and any vector $m=(m_1,\dots,m_r)\in\left(\mathbb{Z}\setminus\lbrace 0\rbrace\right)^r$ of non-zero integers we define the map
\begin{align}\label{mapf}
f_{m,\alpha}\colon X^r\to J,\quad (x_1,\dots,x_r)\mapsto \sum_{j=1}^r m_j(x_j-\alpha)
\end{align}
and we denote by $Z_{m,\alpha}$ the cycle on $J$ obtained by the image of $f_{m,\alpha}$. Philippon \cite{Phi91} and Gubler \cite[(8.7)]{Gub94} introduced the N\'eron--Tate height for higher-dimensional subvarieties of abelian varieties. If $r\le g$, the N\'eron--Tate height of $Z_{m,\alpha}$ with respect to $\mathcal{L}$ is given by
$$h'_{\mathcal{L}}(Z_{m,\alpha})=\frac{\langle \hat{\mathcal{L}}^{r+1}|Z_{m,\alpha}\rangle}{d_K(r+1)\langle \mathcal{L}^r|Z_{m,\alpha}\rangle},$$
where the terms in the brackets denote the (arithmetic) self-intersection numbers of $\mathcal{L}$ equipped with its admissible adelic metric and restricted to $Z_{m,\alpha}$. 

De Jong \cite[Theorem~1.1]{dJo18} showed, that there are rational numbers $a'$, $b'$ and $c'$, such that
$$h'_{\mathcal{L}}(Z_{m,\alpha})=\frac{1}{d_K}\left(a'\cdot\hat{\omega}^2+b'\cdot\varphi(X)+c'\cdot h_{NT}(x_{\alpha})\right),$$
where $\hat{\omega}^2$ denotes the arithmetic self-intersection of the dualizing sheaf $\omega$ of $X$ equipped with its admissible adelic metric, we shortly write $x_{\alpha}=\alpha-\frac{\omega}{2g-2}$ and $h_{NT}(x_{\alpha})$ is the N\'eron--Tate height, which can also be expressed by the arithmetic self-intersection number $h(x_{\alpha})=-\langle\hat{x}_{\alpha},\hat{x}_{\alpha}\rangle$ equipping $x_{\alpha}$ with its admissible adelic metric, see \cite[(5.4)]{Zha93}.
The invariant $\varphi(X)$ is given by
$$\varphi(X)=\sum_{v\in M(K)_0}\varphi(X_v)\log N(v)+\sum_{\sigma\colon K\to \mathbb{C}} \varphi(X_\sigma),$$
where $M(K)_0$ denotes the set of finite places of $K$, $N(v)$ is the cardinality of the residue field at $v$, $X_v$ is a semistable regular model of $X$ over the integers $\mathcal{O}_{K_v}$ of the completion $K_v$ of $K$ at $v$ and $\varphi(X_v)$ is an invariant only depending on the metrized reduction graph of $X_v$ introduced by Zhang \cite[Theorem~1.3.1]{Zha10}. The second sum runs over all embeddings $\sigma\colon K\to \mathbb{C}$, $X_\sigma$ denotes the base change of $X$ induced by $\sigma$ and $\varphi(X_\sigma)$ is an invariant of compact connected Riemann surfaces also introduced in {\em loc.\ cit.}

De Jong computed $a'$, $b'$ and $c'$ in some special cases and gave an algorithm to compute them in general, which was implemented by D. Holmes in SAGE. As an application, de Jong obtained an effective Bogomolov-type result for $Z_{(1),\alpha}$. Further, he remarked, that by the results of the algorithms, a general closed expression for $a'$, $b'$ and $c'$ seems not to be straightforward.

In this paper, we will modify and generalize our combinatorial method in \cite[Section~5.3]{Wil17} to compute arithmetic intersection products on self-products of $X$ in a more general way. Before we give the result in its full generality, we discuss some applications. First, we obtain the following closed expressions for the numbers $a'$, $b'$ and $c'$ above.
\begin{Thm}\label{neron-tate-height}
	Let $X$ be any smooth projective geometrically connected curve of genus $g\ge2$ with semistable reduction over a number field $K$, $\alpha\in \mathrm{Div}^1(X)$ any degree $1$ divisor on $X$ and $m\in (\mathbb{Z}\setminus\lbrace 0\rbrace )^r$ any vector of non-zero integers. If $r=g$, it holds $h'_{\mathcal{L}}(Z_{m,\alpha})=0$. For $1\le r<g$ we have
	$$h'_{\mathcal{L}}(Z_{m,\alpha})=\frac{g-r}{2d_K}\left(a\cdot\hat{\omega}^2+b\cdot\varphi(X)+c\cdot h_{NT}(x_{\alpha})\right)$$
	with $a=m_1^2/(4(g-1)^2)$, $b=0$ and $c=m_1^2/g$ if $r=1$ and
	$$a=\frac{\sum_{j=1}^rm_j^2}{4(g-1)^2}-\frac{(2g+1)\sum_{j<k}^rm_jm_k}{6g(g-1)^2(g-2)},\quad b=\frac{\sum_{j<k}^rm_jm_k}{3g(g-1)(g-2)},\quad c=\frac{\left(\sum_{j=1}^rm_j\right)^2}{g}$$
	if $r\ge 2$.
\end{Thm}

Secondly, we discuss lower bounds for $\hat{\omega}^2$. Zhang \cite[Section~1.4]{Zha10} has obtained, that the lower bound
\begin{align}\label{zhangbound}
\hat{\omega}^2\ge \frac{2g-2}{2g+1}\varphi(X)
\end{align}
would follow from the arithmetic Hodge index conjecture by Gillet-Soul\'e \cite[Conjecture~2]{GS94}, which is not known to be true in general. But for hyperelliptic curves it is known, that the inequality (\ref{zhangbound}) becomes an equality, which shows the sharpness of this conjectured bound.
De Jong obtained the lower bound
\begin{align}\label{dejongbound}
\hat{\omega}^2\ge \frac{2}{3g-1}\varphi(X)
\end{align}
as a consequence of the non-negativity of the N\'eron--Tate height of $Z_{m,\alpha}$, see \cite[Corollary~1.4]{dJo18}, and he remarked, that experiments led to the belief, that this is the best lower bound for $\hat{\omega}^2$ obtainable by this method. Theorem \ref{neron-tate-height} shows, that this is indeed true: We get the best lower bound for $\hat{\omega}^2$ exactly if $\alpha=\frac{\omega}{2g-2}$ and the quotient 
$$\frac{\sum_{j<k}^r m_jm_k}{\sum_{j=1}^r m_j^2}$$
is minimal. This is exactly the case if $\sum_{j=1}^r m_j=0$, which gives the bound (\ref{dejongbound}).
In this paper, we will give stronger bounds as an application of the arithmetic Hodge index theorem for adelic line bundles by Yuan and Zhang \cite[Theorem~3.2]{YZ17}.
\begin{Thm}\label{lowerbound}
	Let $X$ be any smooth projective geometrically connected curve of genus $g\ge2$ with semistable reduction over a number field $K$. It holds
	$$\hat{\omega}^2\ge \frac{g-1}{2g+1}\varphi(X).$$
	For $g=3$ we even have $\hat{\omega}^2\ge \frac{1}{3}\varphi(X)$ and for $g=4$ we have $\hat{\omega}^2\ge \frac{38}{109}\varphi(X)$.
\end{Thm}
It is hard to predict, whether our method can give even stronger bounds.

N\'eron--Tate heights have applications to the study of the generalized Bogomolov conjecture.
We first discuss this in a more general situation. Let $A$ be an abelian variety defined over $K$, $\mathcal{L}$ a symmetric ample line bundle on $A$ defining a principal polarization, $h'_{\mathcal{L}}(\cdot)$ the Néron-Tate height associated to $\mathcal{L}$ and $Z\subseteq A$ a closed subvariety.
Then Zhang \cite[Theorem~1.10]{Zha95} has shown, that the first essential minimum 
$$e'_{\mathcal{L}}(Z)=\sup_{\gf{Y\subset Z}{\mathrm{codim}(Y)=1}} \inf_{x\in (Z\setminus Y)(\overline{K})} h'_{\mathcal{L}}(x)$$
of $Z$ is bounded by
$e'_{\mathcal{L}}(Z)\ge h'_{\mathcal{L}}(Z)$.

The generalized Bogomolov conjecture states, that $e'_{\mathcal{L}}(Z)>0$ if $Z$ is not the translate by a torsion point of an abelian subvariety of $A$. This was first proven by Ullmo \cite{Ull98} if $Z$ is a curve embedded in its Jacobian $A=\mathrm{Pic}^0(Z)$ and then by Zhang \cite{Zha98} in the general case. The effective Bogomolov conjecture asks for an effective positive lower bound for $e'_{\mathcal{L}}(Z)$.
As a direct consequence of Theorems \ref{neron-tate-height} and \ref{lowerbound} we can deduce a lower bound for $e'_{\mathcal{L}}(Z_{m,\alpha})$.
\begin{Cor}	Let $X$ be any smooth projective geometrically connected curve of genus $g\ge2$ with semistable reduction over a number field $K$, $\alpha\in \mathrm{Div}^1(X)$ any degree $1$ divisor on $X$ and $m\in (\mathbb{Z}\setminus\lbrace 0\rbrace )^r$ any vector of non-zero integers.
	If $r=1$, it holds $h'_{\mathcal{L}}(Z_{m,\alpha})\ge \frac{m_1^2}{8d_K(2g+1)}\varphi(X)>0$. For $2\le r< g$ we have
	$$h'_{\mathcal{L}}(Z_{m,\alpha})\ge (g-r)\frac{(3g^2-8g-1)\sum^r_{j=1}m_j^2+(2g+1)\left(\sum_{j=1}^r m_j\right)^2}{24g(g-1)(g-2)(2g+1)d_K}\varphi(X)>0.$$
\end{Cor}
Thus, we obtain an effective Bogomolov-type result for all $Z_{m,\alpha}$ if $r<g$.
Note, that we indeed have $\varphi(X)>0$: For any embedding $\sigma\colon K\to \mathbb{C}$ we obtain the positivity $\varphi(X_\sigma)>0$ from \cite[2.5~Remark~1]{Zha10} and for $v\in M(K)_0$ it was shown by Cinkir \cite[Theorem~2.11]{Cin11}, that we have $\varphi(X_v)\ge 0$. More precisely, Cinkir proved that we have
$$\varphi(X_v)\ge \frac{g-1}{2g(7g+5)}\delta_0(X_v)+\sum_{j=1}^{\lfloor g/2\rfloor}\frac{2j(g-j)}{g}\delta_j(X_v),$$
where $\delta_0(X_v)$ denotes the number of non-separating geometric double points on the reduction of $X$ at $v$ and for $1\le j\le \lfloor g/2\rfloor$ we write $\delta_j(X_v)$ for the number of geometric double points on the reduction of $X$ at $v$ such that the local normalization has two connected components, one of arithmetic genus $j$ and one of arithmetic genus $g-j$.
Hence, we obtain a very explicit lower bound for $e'_{\mathcal{L}}(Z_{m,\alpha})$.

Next, we give our general statement on the arithmetic intersection numbers on $X^r$.
Let $\hat{\Delta}$ be the admissible adelic line bundle on $X^2$ associated to the diagonal in $X^2$. We define the following modified versions of $\hat{\omega}$ and $\hat{\Delta}$
\begin{align}\label{modification}
\hat{\omega}^{\alpha}&=\hat{\omega}+2\hat{\alpha}-\pi^*\hat{\alpha}^2\\
\hat{\Delta}^{\alpha}&=\hat{\Delta}-p_1^*\hat{\alpha}-p_2^*\hat{\alpha}+\pi_2^*\hat{\alpha}^2.\nonumber
\end{align}
Here, $\hat{\alpha}$ is equipped with its admissible adelic metric, $\hat{\alpha}^2$ denotes its arithmetic self-intersection number, we write $\pi\colon X\to \mathrm{Spec}~K$ and $\pi_2\colon X^2\to \mathrm{Spec}~K$ for the structure morphisms and we write in general $p_{j_1\dots j_l}\colon X^r\to X^l$ for the projection to the $j_1$-th,$\dots$, $j_l$-th factors. Further, we denote the following adelic line bundles on $X^r$
\begin{align}\label{deltajk}
\hat{\Delta}^{\alpha}_{jk}=\begin{cases} -p_j^*\hat{\omega}^{\alpha} & \text{if } j=k,\\ p_{jk}^*\hat{\Delta}^{\alpha}& \text{if }j\neq k .\end{cases}
\end{align}

Our main result computes the arithmetic intersection numbers of adelic $\mathbb{Q}$-line bundles given by linear combinations of the $\hat{\Delta}^{\alpha}_{jk}$'s in terms of the self-intersection number $\hat{\omega}^2$, Zhang's invariant $\varphi(X)$ and the N\'eron--Tate height of $x_{\alpha}$.

\begin{Thm}\label{mainthm}
	Let $X$ be any smooth projective geometrically connected curve of genus $g\ge2$ with semistable reduction over a number field $K$, $r>0$ an integer and $\overline{\mathcal{M}}_1,\dots \overline{\mathcal{M}}_{r+1}$ adelic $\mathbb{Q}$-line bundles on $X^r$ of the form
	$$\overline{\mathcal{M}}_l=\tfrac{1}{2}\sum_{j,k=1}^r t_{l,j,k}\hat{\Delta}^{\alpha}_{jk},$$
	where $t_{l,j,k}\in\mathbb{Q}$ are any rational numbers satisfying $t_{l,j,k}=t_{l,k,j}$ for all $l,j,k$. 
	\begin{enumerate}[(a)]
		\item
		The intersection number $\langle \mathcal{M}_1,\dots,\mathcal{M}_r\rangle$ is given by
		$$\langle \mathcal{M}_1,\dots,\mathcal{M}_r\rangle=\sum_{\tau\in \mathcal{S}_r}\sum_{\pi\in \Pi_r}(-g)^{|\pi|}\prod_{B\in\pi}\frac{1}{|B|}\sum_{\sigma\colon \mathbb{Z}/|B|\xrightarrow{\sim}B} \prod_{j=0}^{|B|-1}t_{\tau(\sigma(j)),\sigma(j),\sigma(j+1)},$$
		where $\mathcal{S}_r$ denotes the symmetric group of $\lbrace 1,\dots,r\rbrace$ and $\Pi_r$ is the set of all partitions of $\lbrace 1,\dots,r\rbrace$.
		\item
		The arithmetic intersection number $\langle\overline{\mathcal{M}}_1,\dots,\overline{\mathcal{M}}_{r+1}\rangle$ is given by
		$$\frac{3gc_1+(2g+1)c_3}{24(g-1)}\hat{\omega}^2-\frac{c_3}{12}\cdot \varphi(X)+\frac{(g-1)(c_1+c_3-(g-1)c_2)}{2}h_{NT}(x_\alpha)$$
		where
		\begin{align*}
		c_1=&\sum_{\tau\in\mathcal{S}_{r+1}}\sum_{\pi\in \Pi_{r}}(-g)^{|\pi|-1}\sum_{B'\in\pi}\left(\prod_{B\in \pi\setminus\lbrace B'\rbrace}\frac{1}{|B|}\sum_{\sigma\colon \mathbb{Z}/|B|\xrightarrow{\sim}B}\prod_{j=0}^{|B|-1}t_{\tau(\sigma(j)),\sigma(j),\sigma(j+1)}\right)\\
		&\times\sum_{k\in B'}\sum_{\sigma\colon \mathbb{Z}/|B'|\xrightarrow{\sim} B'}t_{\tau(\sigma(0)),\sigma(0),k}t_{\tau(r+1),\sigma(1),k}\prod_{j=1}^{|B'|-1}t_{\tau(\sigma(j)),\sigma(j),\sigma(j+1)},
		\end{align*}
		\begin{align*}
		c_2&=\sum_{\tau\in\mathcal{S}_{r+1}}\sum_{\pi\in \Pi_{r}}(-g)^{|\pi|-1}\sum_{B'\in\pi}\left(\prod_{B\in \pi\setminus\lbrace B'\rbrace}\frac{1}{|B|}\sum_{\sigma\colon \mathbb{Z}/|B|\xrightarrow{\sim}B}\prod_{j=0}^{|B|-1}t_{\tau(\sigma(j)),\sigma(j),\sigma(j+1)}\right)\\
		&\times\sum_{1\le j<k\le |B'|}\sum_{\sigma\colon \mathbb{Z}/|B'|\xrightarrow{\sim} B'}t_{\tau(\sigma(0)),\sigma(0),\sigma(k)}t_{\tau(r+1),\sigma(1),\sigma(j)}\prod_{j=1}^{|B'|-1}t_{\tau(\sigma(j)),\sigma(j),\sigma(j+1)}
		\end{align*}
		and
		\begin{align*}
		c_3&=\sum_{\tau\in\mathcal{S}_{r+1}}\sum_{\pi\in \Pi_{r}}(-g)^{|\pi|-1}\sum_{B'\in\pi}\left(\prod_{B\in \pi\setminus\lbrace B'\rbrace}\frac{1}{|B|}\sum_{\sigma\colon \mathbb{Z}/|B|\xrightarrow{\sim}B}\prod_{j=0}^{|B|-1}t_{\tau(\sigma(j)),\sigma(j),\sigma(j+1)}\right)\\
		&\times\sum_{1\le j<k\le |B'|}\sum_{\sigma\colon \mathbb{Z}/|B'|\xrightarrow{\sim} B'}t_{\tau(\sigma(0)),\sigma(0),\sigma(j)}t_{\tau(r+1),\sigma(1),\sigma(k)}\prod_{j=1}^{|B'|-1}t_{\tau(\sigma(j)),\sigma(j),\sigma(j+1)}.
		\end{align*}
	\end{enumerate}
\end{Thm}
The proof of the theorem is based on our method of associating graphs to Deligne pairings introduced in \cite[Section~5.3]{Wil17}. The modification in Equation (\ref{modification}) makes the combinatorics simpler, which makes it possible to prove the theorem.

To obtain Theorem \ref{neron-tate-height} from Theorem \ref{mainthm} we will show, that it holds
$$f_{m,\alpha}^*\hat{\mathcal{L}}=-\frac{1}{2}\sum_{j,k=1}^r m_jm_k\hat{\Delta}^{\alpha}_{jk}.$$
Hence by the projection formula, Theorem \ref{mainthm} can be applied to compute the N\'eron--Tate heights in Theorem \ref{neron-tate-height} and it is left as a combinatorial exercise to express them in the simplified form.

As an application of the arithmetic Hodge index theorem for adelic line bundles by Yuan--Zhang \cite[Theorem~3.2]{YZ17}, we will prove the following result.
\begin{Thm}\label{generallowerbound}
	Let $X$ be any smooth projective geometrically connected curve of genus $g\ge2$ with semistable reduction over a number field $K$, $r>0$ an integer and $m\in\left(\mathbb{Z}\setminus\lbrace0\rbrace\right)^r$ any vector of non-zero integers.
	For any rational numbers $t_{jk}\in \mathbb{Q}$ with $1\le j,k\le r$ satisfying 
	$$g\sum_{j=1}^r t_{jj}=\sum_{j\neq k}^r t_{jk},$$
	define the adelic $\mathbb{Q}$-line bundle $\overline{\mathcal{M}}=\sum_{j,k=1}^r t_{jk}m_j m_k\hat{\Delta}^{\alpha}_{jk}$ on $X^r$.
	Then it holds
	$$\langle (f_{m,\alpha}^*\hat{\mathcal{L}})^{r-1},\overline{\mathcal{M}}^2\rangle\le 0.$$
\end{Thm}
We remark, that for $r\ge g+2$ we even have $\langle (f_{m,\alpha}^*\hat{\mathcal{L}})^{r-1},\overline{\mathcal{M}}^2\rangle=0$.
We will deduce Theorem \ref{lowerbound} from Theorem \ref{generallowerbound} by applying it to suitable choices of $r$, $m$, and $t_{jk}$.

\subsubsection*{Outline}
In Section \ref{secadelic} we recall the required facts on adelic line bundles. We discuss Deligne pairings and their relations to intersection numbers in Section \ref{secdeligne}. In Section \ref{secadmissible} we recall the definition of admissible metrics on abelian varieties and on curves.

In Section \ref{secgraphs} we associate a graph to certain intersection numbers and vice versa, an intersection number to any graph. Further, we express the intersection number associated to a graph in terms of usual invariants of the graph by several reduction steps. 
In the subsequent section we prove Theorem \ref{mainthm} using this method. We study the self-intersection numbers of $f_{m,\alpha}^*\hat{\mathcal{L}}$ as a special case of Theorem \ref{mainthm} in Section \ref{secnerontate}. In particular, we obtain a proof for Theorem \ref{neron-tate-height}. Finally, in the last section we deduce Theorems \ref{lowerbound} and \ref{generallowerbound} from the arithmetic Hodge index theorem for adelic line bundles.

\subsubsection*{Terminology}
We always denote by $K$ a number field of degree $d_K=[K:\mathbb{Q}]$ and we write $\mathcal{O}_K$ for its ring of integers.
We set $M(K)=M(K)_0\cup M(K)_\infty$, where $M(K)_0$ is the set of finite places of $K$ and $M(K)_\infty$ denotes the set of complex embeddings $K\to \mathbb{C}$. For $v\in M(K)_0$ we write $K_v$ for the completion of $K$ with respect to $v$ and we fix an algebraic closure $\overline{K}_v$ with ring of integers $\mathcal{O}_{\overline{K}_v}$. For $v\in M(K)_\infty$ we denote $K_v$ for the completion of the image of $v$ in $\mathbb{C}$ and $\overline{K}_v=\mathbb{C}$. We fix a collection of absolute values $|\cdot|_v$ on $K_v$ for all $v\in M(K)$, which satisfies the product formula.

Throughout this paper, $\pi\colon X\to \mathrm{Spec}(K)$ will always be a smooth projective geometrically connected curve of genus $g\ge 2$ with semistable reduction over $\mathcal{O}_K$. We write $\omega$ for the canonical bundle of $X$ and we denote $J=\mathrm{Pic}^0(X)$ for the Jacobian of $X$. We fix an ample symmetric line bundle $\mathcal{L}$ on $J$, which induces the canonical principal polarization of $J$ and which is rigidified at the origin.

We will always denote by $r$ a positive integer, $\pi_r\colon X^r\to \mathrm{Spec}(K)$ for the structure morphism and for $j_1,\dots,j_l\le r$ we write $p_{j_1,\dots,j_l}\colon X^r\to X^l$ for the projection to the $j_1$-th, $\dots$, $j_l$-th factors. Furthermore, $p^k\colon X^r\to X^{r-1}$ denotes the projection forgetting the $k$-th factor.

Moreover, we will always denote by $m\in\left(\mathbb{Z}\setminus\lbrace 0\rbrace\right)^r$ an $r$-dimensional vector of non-zero integers. Further, $\alpha$ will always be a line bundle of degree $1$ on $X$. We will often choose $\alpha=\frac{\omega}{2g-2}$, although this is only defined up to a torsion element in $J$. Also, we shortly write $x_\alpha=\alpha-\frac{\omega}{2g-2}$. We denote the morphism $f_{m,\alpha}$ as in (\ref{mapf}) and we denote $Z_{m,\alpha}$ for the cycle on $J$ induced by its image.

\section{Adelic line bundles}\label{secadelic}
In this section we recall the preliminaries on adelic line bundles. The main reference is \cite[Section~1]{Zha95}, see also \cite[Section~2]{dJo18}. Let $Y$ be a smooth projective variety over $K$ and $\mathcal{M}$ a line bundle on $Y$. For $v\in M(K)$ we denote $Y_v$ and $\mathcal{M}_v$ for the pullbacks of $Y$ and $\mathcal{M}$ induced by the canonical map $K\to \overline{K}_v$. A metric on $\mathcal{M}_v$ is by definition a collection of $\overline{K}_v$-norms on each fibre $y^*\mathcal{M}_v$ for $y\in Y_v(\overline{K}_v)$.

If $v\in M(K)_0$ is a finite place, we obtain a natural metric $\|\cdot\|_{\widetilde{\mathcal{M}_v}}$ of $\mathcal{M}_v$ for any positive integer $e$ and any projective flat model $(\widetilde{Y_v},\widetilde{\mathcal{M}_v})$ of $(Y_v,\mathcal{M}_v^{\otimes e})$ over $\mathrm{Spec}(\mathcal{O}_{\overline{K}_v})$. This metric is given as follows: For any point $y\in Y_v(\overline{K}_v)$, with its unique extension $\widetilde{y}\in \widetilde{Y_v}(\mathcal{O}_{\overline{K}_v})$, and any $m\in y^*\mathcal{M}_v$ we put 
$$\|m\|_{\widetilde{\mathcal{M}_v}}=\inf_{a\in \overline{K}_v}\lbrace |a|_v^{1/e}: m\in a \widetilde{y}^*\widetilde{\mathcal{M}_v}\rbrace.$$
In general, a metric $\|\cdot\|$ on $\mathcal{M}_v$ is called continuous and bounded, if there exists a model $(\widetilde{Y_v},\widetilde{\mathcal{M}_v})$ of $(Y_v,\mathcal{M}_v)$, such that $\log\frac{\|\cdot\|}{\|\cdot\|_{\widetilde{\mathcal{M}_v}}}$ is continuous and bounded.

An adelic metric for $\mathcal{M}$ is a collection of metrics $\|\cdot\|=\lbrace\|\cdot\|_{v}~|~v\in M(K)\rbrace$ such that each $\|\cdot\|_v$ is bounded, continuous and $\mathrm{Gal}(\overline{K}_v/K_v)$ invariant and there exist a non-empty open subset $U\subseteq \mathrm{Spec}(\mathcal{O}_K)$ and a model $(\widetilde{Y},\widetilde{\mathcal{M}})$ of $(Y,\mathcal{M})$ over $U$, such that for every $v\in U$ we have $\|\cdot\|_v=\|\cdot\|_{\widetilde{\mathcal{M}}_v}$, where $\widetilde{\mathcal{M}}_v$ is the pullback of $\widetilde{\mathcal{M}}$ induced by $\mathrm{Spec}(\mathcal{O}_{\overline{K}_v})\to U$. We call an adelic metric $\|\cdot\|$ semipositive, if there is a sequence of models $(\widetilde{Y}_n,\widetilde{\mathcal{M}}_n)$ of $(Y,\mathcal{M})$ over $\mathrm{Spec}(\mathcal{O}_K)$ with $\widetilde{\mathcal{M}}_n$ relatively semipositive on $\widetilde{Y}_n$, such that for all $v\in M(K)_{0}$ the series $\log\frac{\|\cdot\|_{\widetilde{\mathcal{M}}_{n.v}}}{\|\cdot\|_v}$ converges to $0$ uniformly in $Y(\overline{K}_v)$ and for each $v\in M(K)_\infty$ the metric $\|\cdot\|_v$ is a smooth hermitian metric with semipositive curvature form.
We call an adelic metrized line bundle $\overline{\mathcal{M}}=(\mathcal{M},\|\cdot\|)$ integrable if there are two semipositive metrized line bundles $\overline{\mathcal{M}}_1$ and $\overline{\mathcal{M}}_2$ and an isometry $\overline{\mathcal{M}}\cong \overline{\mathcal{M}}_1\otimes\overline{\mathcal{M}}_2^{\otimes-1}$.

Zhang has shown in \cite[(1.5)]{Zha95}, that the arithmetic intersection product by Gillet and Soul\'e \cite{GS90} can be extended to integrable metrized line bundles. Hence, for integrable line bundles $\overline{\mathcal{M}}_1,\dots,\overline{\mathcal{M}}_r$ and an integral cycle $Z$ on $Y$ of dimension $d_Z=\dim Z$ we obtain a symmetric and multi-linear intersection number
$$\langle \overline{\mathcal{M}}_1,\dots,\overline{\mathcal{M}}_r|Z\rangle,$$
which is zero for $d_Z\notin\lbrace r-1,r\rbrace$ and for $d_Z=r$ we have
$$\langle \overline{\mathcal{M}}_1,\dots,\overline{\mathcal{M}}_r|Z\rangle=\langle \mathcal{M}_1,\dots,\mathcal{M}_r|Z\rangle=c_1(\mathcal{M}_1)\dots c_1(\mathcal{M}_r)[Z]\in\mathbb{Z}.$$
We shortly write $\langle \overline{\mathcal{M}}^{s},\overline{\mathcal{N}}^{r-s}|Z\rangle=\langle \overline{\mathcal{M}}_1,\dots,\overline{\mathcal{M}}_r|Z\rangle$ if $\overline{\mathcal{M}}_1=\dots=\overline{\mathcal{M}}_s=\overline{\mathcal{M}}$ and $\overline{\mathcal{M}}_{s+1}=\dots=\overline{\mathcal{M}}_r=\overline{\mathcal{N}}$. Further, we write $\langle \overline{\mathcal{M}}_1,\dots,\overline{\mathcal{M}}_r\rangle=\langle \overline{\mathcal{M}}_1,\dots,\overline{\mathcal{M}}_r|Z\rangle$ if $Z=Y$.
If we denote the space of integrable line bundles by $\widehat{\mathrm{Pic}}(X)$, then an integrable $\mathbb{Q}$-line bundle is an element in $\widehat{\mathrm{Pic}}(X)\otimes_\mathbb{Z}\mathbb{Q}$. By multi-linearity the arithmetic intersection product also extends to integrable $\mathbb{Q}$-line bundles.

Next, we define the notion of nefness for integrable line bundles. Let $\widetilde{Y}$ be a model of $Y$ over $\mathcal{O}_K$. A hermitian line bundle on $\widetilde{Y}$ is a pair $\overline{\mathcal{M}}=(\mathcal{M},\|\cdot\|)$ consisting of a line bundle $\mathcal{M}$ on $\widetilde{Y}$ and a collection of smooth hermitian metrics $\|\cdot\|=\lbrace\|\cdot\|_{v}~|~v\in M(K)_{\infty}\rbrace$ on the pullbacks of $\mathcal{M}$ induced by the $v$'s in $M(K)_{\infty}$, which is invariant under the action of the complex conjugation. A hermitian line bundle $\overline{\mathcal{M}}$ is called nef if for all irreducible closed subvarieties $\mathcal{Z}\subseteq \widetilde{Y}$ the arithmetic intersection product $\langle \overline{\mathcal{M}}^{\dim\mathcal{Z}}|\mathcal{Z}\rangle$ is non-negative and if $\|\cdot\|_v$ has a semipositive curvature form for all $v\in M(K)_{\infty}$.

Now let $\overline{\mathcal{M}}$ be an integrable line bundle on $Y$. We call $\overline{\mathcal{M}}$ nef if there is a sequence of models $(\widetilde{Y}_n,\widetilde{\mathcal{M}}_n)$ of $(Y,\mathcal{M})$ over $\mathrm{Spec}(\mathcal{O}_K)$, where the $\widetilde{\mathcal{M}}_n$'s are now nef hermitian line bundles, such that the adelic model metrics of $(\widetilde{Y}_n,\widetilde{\mathcal{M}}_n)$ uniformly converge to the adelic metric of $\overline{\mathcal{M}}$.

We define the height of an integral cycle $Z$ on $Y$ of dimension $d_Z=\dim Z$ with respect to an integrable ample line bundle $\overline{\mathcal{M}}$ to be
\begin{align}\label{height}
h'_{\overline{\mathcal{M}}}(Z)=\frac{\langle \overline{\mathcal{M}}^{d_Z+1}|Z\rangle}{d_K(d_Z+1)\langle \mathcal{M}^{d_Z}|Z\rangle}.
\end{align}

We will use the following projection formula. Let $f\colon Y\to S$ be a morphism of smooth projective varieties over $K$ and $\overline{\mathcal{M}}_1,\dots,\overline{\mathcal{M}}_r$ integrable line bundles on $S$. Further, let $Z$ be an integral cycle on $Y$. Then it is shown in \cite[Proposition~3.1]{dJo18}, that
\begin{align}\label{projectionformula}
\langle \overline{\mathcal{M}}_1,\dots,\overline{\mathcal{M}}_r|f_*(Z)\rangle=\langle f^*\overline{\mathcal{M}}_1,\dots,f^*\overline{\mathcal{M}}_r|Z\rangle
\end{align}
If we further set $d_Y=\dim Y$ and $d_S=\dim S$ and choose an integrable line bundle $\overline{\mathcal{N}}$ on $Y$, then we have
\begin{align}\label{vanishingbydimension}
\langle(f^*\overline{\mathcal{M}}_1)^{r+1},\overline{\mathcal{N}}^{d_Y-r}\rangle=0 \quad \text{and} \quad \langle(f^*\mathcal{M}_1)^r,\mathcal{N}^{d_Y-r}\rangle=0
\end{align}
for $r>d_S$. If the metrics are induced by models, the first equation follows from \cite[Section~4.4]{GS90}, as the arithmetic cycle $\overline{\mathcal{M}}_1^{r+1}$ is trivial. In general it follows by taking limits. The second equation follows by classical intersection theory.

There is another projection formula.
Let $S$ and $T$ be smooth projective varieties over $K$, $\overline{\mathcal{M}}_1,\dots,\overline{\mathcal{M}}_s$ integrable line bundles on $S$ and $\overline{\mathcal{N}}_1,\dots,\overline{\mathcal{N}}_t$ integrable line bundles on $T$. Denote $p\colon S\times T\to S$ and $q\colon S\times T\to T$ for the two projections. It is shown in \cite[Proposition~3.2]{dJo18}, that we have
\begin{align}\label{projection}
\langle p^*\overline{\mathcal{M}}_1,\dots, p^*\overline{\mathcal{M}}_s,q^*\overline{\mathcal{N}}_1,\dots, q^*\overline{\mathcal{N}}_t|S\times T\rangle=\langle\overline{\mathcal{M}}_1,\dots,\overline{\mathcal{M}}_s|S\rangle \langle\overline{\mathcal{N}}_1,\dots, \overline{\mathcal{N}}_t|T\rangle.
\end{align}
\section{Deligne pairings}\label{secdeligne}
In this section we collect some facts about Deligne pairings. Details can be found in \cite[Section~1.1]{Zha96} as well as in \cite[Section~3]{dJo18}.
Let $f\colon Y\to S$ be a smooth morphism of smooth projective varieties over $K$ of relative dimension $n$ and $\overline{\mathcal{M}}_0,\dots,\overline{\mathcal{M}}_n$ semipositive adelic line bundles on $Y$. The Deligne pairing $\langle \overline{\mathcal{M}}_0,\dots,\overline{\mathcal{M}}_n\rangle$ is a integrable line bundle on $S$. Its underlying line bundle is locally generated by symbols $\langle m_0,\dots,m_n\rangle$, where $m_i$ is a local section of $\mathcal{M}_i$, and for any $0\le i\le n$ and any function $h$ on $Y$, such that the intersection $\prod_{j\neq i}\mathrm{div}(m_j)=\sum_{j}n_j Z_j$ is finite over $S$ and has empty intersection with $\mathrm{div}(h)$, we have the relation
$$\langle m_0,\dots,m_{i-1},h m_i,m_{i+1},\dots,m_n\rangle=\prod_{j} \mathrm{Norm}_{Z_j/S}(h)^{n_j}\langle m_0,\dots,m_n\rangle.$$

The metrics of $\langle \overline{\mathcal{M}}_0,\dots,\overline{\mathcal{M}}_n\rangle$ can be described recursively. We may assume $S=\mathrm{Spec}(K)$. For any non-zero local sections $m_0,\dots,m_n$ of $\overline{\mathcal{M}}_0,\dots,\overline{\mathcal{M}}_n$ with empty common zero locus, such that $\mathrm{div}(m_n)$ is a prime divisor on $X$ and $m_i|_{\mathrm{div}(m_n)}$ is non-zero for all $i$, we have
\begin{align}\label{delignerecursive}
\log \|\langle m_0,\dots,m_n\rangle\|_v=&\log\|\langle m_0|_{\mathrm{div}(m_n)},\dots,m_{n-1}|_{\mathrm{div}(m_n)}\rangle\|_v\\
&+\int_{Y(\bar{K}_v)}\log\|m_n\|_v c_1(\overline{\mathcal{M}}_0)\dots c_1(\overline{\mathcal{M}}_{n-1})\nonumber
\end{align}
for all $v\in M(K)$. We have to make sense of the integral if $v\in M(K)_0$.
It is enough to do this, if the adelic line bundles $\overline{\mathcal{M}}_0,\dots,\overline{\mathcal{M}}_n$ are given by models $(\widetilde{Y},\widetilde{\mathcal{M}}_j)$ of $(Y,\mathcal{M}_j^{\otimes e})$ over $\mathrm{Spec}(\mathcal{O}_K)$. In general, one has to take limits. 
If we write $\widetilde{m}_n$ for the section of $\widetilde{\mathcal{M}}_n$ extending the section $m_n^{\otimes e}$ of $\mathcal{M}_n^{\otimes e}$, then $V=\mathrm{div}(\widetilde{m}_n)-e\cdot\overline{\mathrm{div}(m_n)}$ is a Weil divisor $V=\sum_{v\in M(K)_0} V_v$ supported in the closed fibres of $\widetilde{Y}$. The integral is defined to be
\begin{align*}
\int_{Y(\bar{K}_v)}\log\|m_n\|_v c_1(\overline{\mathcal{M}}_0)\dots c_1(\overline{\mathcal{M}}_{n-1})=c_1(\widetilde{\mathcal{M}}_0)\dots c_1(\widetilde{\mathcal{M}}_{n-1})[V_v]\log N(v)/e^n.
\end{align*}
It turns out, that $\langle \overline{\mathcal{M}}_0,\dots,\overline{\mathcal{M}}_n\rangle$ with this metric is indeed a integrable line bundle on $S$. For any integrable $\mathbb{Q}$-line bundles $\overline{\mathcal{M}}_0,\dots,\overline{\mathcal{M}}_n$ on $Y$ we also obtain an integrable $\mathbb{Q}$-line bundle $\langle \overline{\mathcal{M}}_0,\dots,\overline{\mathcal{M}}_n\rangle$ on $S$ by multi-linearity.

If $S=\mathrm{Spec}(K)$, the Deligne pairing is given by the intersection number
$$\langle \overline{\mathcal{M}}_0,\dots,\overline{\mathcal{M}}_n|Y\rangle=\langle \langle \overline{\mathcal{M}}_0,\dots,\overline{\mathcal{M}}_n\rangle|\mathrm{Spec}(K)\rangle.$$
In general, one has for integrable $\mathbb{Q}$-line bundles $\overline{\mathcal{M}_0},\dots,\overline{\mathcal{M}}_n$ on $Y$ and integrable $\mathbb{Q}$-line bundles $\overline{\mathcal{N}_1},\dots,\overline{\mathcal{N}}_r$ on $S$ the identity
\begin{align}\label{deligne1}
\langle f^*(\overline{\mathcal{N}}_1),\dots,f^*(\overline{\mathcal{N}}_r),\overline{\mathcal{M}}_0,\dots,\overline{\mathcal{M}}_n|Y\rangle=\langle\overline{\mathcal{N}}_1,\dots,\overline{\mathcal{N}}_r,\langle \overline{\mathcal{M}}_0,\dots,\overline{\mathcal{M}}_n\rangle|S\rangle
\end{align}
and the identity
\begin{align}\label{delgine2}
\langle f^*(\overline{\mathcal{N}}_1),\dots,f^*(\overline{\mathcal{N}}_r),\overline{\mathcal{M}}_1,\dots,\overline{\mathcal{M}}_n|Y\rangle=c_1(\mathcal{M}_1)\dots c_1(\mathcal{M}_n)[f]\langle \overline{\mathcal{N}}_1,\dots,\overline{\mathcal{N}}_r|S\rangle,
\end{align}
where $c_1(\mathcal{M}_1)\dots c_1(\mathcal{M}_n)[f]$ denotes the multidegree of the generic fibre of $f$ with respect to the line bundles $\mathcal{M}_1,\dots,\mathcal{M}_n$. This follows from \cite[Proposition~3.5]{dJo18}.
\section{admissible metrics}\label{secadmissible}
Continuing the notation of the introduction, we recall the admissible adelic metrics for the line bundle $\mathcal{L}$ on $J$, for the canonical bundle $\omega$ on $X$ and for the diagonal bundle $\Delta=\mathcal{O}_{X^2}(\Delta)$ on $X^2$. Details can be found in \cite{Zha95} and \cite[Section~4]{dJo18}.

As $\mathcal{L}$ is a symmetric bundle and rigidified at the origin, there is a unique choice of an isomorphism $\phi\colon\mathcal{L}^{\otimes 4}\xrightarrow{\sim} [2]^*\mathcal{L}$, where $[2]\colon J\to J$ denotes the multiplication by $2$. It was shown by Zhang \cite[Theorem~2.2]{Zha95}, that there is a unique integrable adelic metric $(\|\cdot\|_v)_{v\in M(K)}$ on $\mathcal{L}$, such that for each $v\in M(K)$ the isomorphism $\phi$ becomes an isometry. We call this metric the admissible metric and we write $\hat{\mathcal{L}}$ for the corresponding integrable line bundle. 
For any closed subvariety $Z$ of $J$ the height $h'_{\mathcal{L}}(Z)=h'_{\hat{\mathcal{L}}}(Z)$ defined in Equation (\ref{height}) coincides with the N\'eron-Tate height defined by Philippon \cite{Phi91} and Gubler \cite[(8.7)]{Gub94}, see \cite[Section~3.1]{Zha95}. We will need, that $\hat{\mathcal{L}}$ is nef in the sense of Section \ref{secadelic}.

\begin{Lem}\label{lemnef}
	The integrable line bundle $\hat{\mathcal{L}}$ is nef.
\end{Lem}
\begin{proof}
The construction in \cite[(2.3)]{Zha95} gives a sequence of models $(\widetilde{J}_n,\widetilde{\mathcal{L}}_n)$ of $(J,\mathcal{L}^{\otimes e_n})$ over $\mathrm{Spec}(\mathcal{O}_K)$ for some $e_n$, such that $\widetilde{\mathcal{L}}_n$ is relatively ample.
For every $n$ and every $v\in M(K)_\infty$ we choose the hermitian metric on $\widetilde{\mathcal{L}}_{n,v}\cong\mathcal{L}_{v} $ to be the unique metric, such that the pullback of the isomorphism $\phi\colon\mathcal{L}^{\otimes 4}\xrightarrow{\sim} [2]^*\mathcal{L}$ by $v\colon K\to \mathbb{C}$ is an isometry. By \cite[Proposition II.2.1]{Mor85} this metric can also be described as the unique metric, such that the rigidification at the orign is an isometry and its curvature form is translation-invariant and therefore positive. We denote $\overline{\mathcal{L}}_n=(\mathcal{L},\|\cdot\|_n)$ for the adelic line bundle induced by the model $(\widetilde{J}_n,\widetilde{\mathcal{L}}_n)$ and equipped with the hermitian metrics above.
Then the sequence of the adelic metrics of $\overline{\mathcal{L}}_n$ uniformly converges to the adelic metric of $\hat{\mathcal{L}}$.

Let now $\mathcal{Z}\subseteq \widetilde{J}_n$ be an irreducible closed subvariety. As $\widetilde{\mathcal{L}}_n$ is ample and $c_1(\widetilde{\mathcal{L}}_n)$ is positive we know by \cite[Proposition 3.2.4 Remark (iii)]{BGS94} that there exists a constant $C_n\ge 0$ independent of $\mathcal{Z}$, such that
$$\langle \widetilde{\mathcal{L}}_n^{\dim \mathcal{Z}}|\mathcal{Z}\rangle\ge -C_n\cdot \deg_{\mathcal{L}^{\otimes e_n}}(\mathcal{Z}).$$
We may choose $C_n=\max\left\lbrace-\inf_{\mathcal{Z} \subseteq \widetilde{J}_n} \frac{\langle \widetilde{\mathcal{L}}_n^{\dim \mathcal{Z}}|\mathcal{Z}\rangle}{\deg_{\mathcal{L}^{\otimes e_n}}(\mathcal{Z})},0\right\rbrace$.
If $\mathcal{Z}$ is contained in the fiber of a closed point of $\mathrm{Spec}(\mathcal{O}_K)$, we always have 
$\langle \widetilde{\mathcal{L}}_n^{\dim \mathcal{Z}}|\mathcal{Z}\rangle\ge0$ by the ampleness of $\widetilde{\mathcal{L}}_n$. Otherwise, $\mathcal{Z}$ is the closure in $\widetilde{J}_n$ of an irreducible cycle $Z\subseteq J$. If we fix $Z\subseteq J$ and consider for varying $n$ the closure $\mathcal{Z}_n$ in $\widetilde{J}_n$ of $Z\subseteq J$, we obtain
$$\lim_{n\to \infty} \frac{\langle \widetilde{\mathcal{L}}_n^{\dim \mathcal{Z}_n}|\mathcal{Z}_n\rangle}{e_n^{\dim\mathcal{Z}_n}}=\lim_{n\to \infty}\langle \overline{\mathcal{L}}_n^{\dim Z+1}|Z\rangle=\langle \hat{\mathcal{L}}^{\dim Z+1}|Z\rangle\ge 0$$
by \cite[Theorem 2.4 (b)]{Zha95}. Hence, the uniform convergence of $\overline{\mathcal{L}}_n$ implies that $\lim_{n\to\infty}C_n= 0$. This can be checked by formula (\ref{delignerecursive}), respectively by its global version \cite[Eq. (3.3)]{dJo18}, and induction on $\dim Z$.

If we write $\widetilde{\mathcal{M}}_n$ for the hermitian line bundle obtained from $\widetilde{\mathcal{L}}_n$ by multiplying the hermitian metric by the constant $\exp(-2C_n/[K:\mathbb{Q}])$, we can deduce from \cite[Proposition 3.2.2]{BGS94} that
\begin{align*}
	\langle \widetilde{\mathcal{M}}_n^{\dim \mathcal{Z}}|\mathcal{Z}\rangle&=\langle \widetilde{\mathcal{L}}_n^{\dim \mathcal{Z}}|\mathcal{Z}\rangle+C_n\cdot \dim\mathcal{Z}\cdot \deg_{\mathcal{L}^{\otimes e_n}}(\mathcal{Z})\\
	&\ge C_n (\dim \mathcal{Z}-1)\cdot \deg_{\mathcal{L}^{\otimes e_n}}(\mathcal{Z})\ge 0
\end{align*}
if $\dim \mathcal{Z}\ge1$. If $\dim \mathcal{Z}=0$, we have $\langle \widetilde{\mathcal{M}}_n^{\dim \mathcal{Z}}|\mathcal{Z}\rangle\ge 0$ by \cite[Equation (3.1.4)]{BGS94}. We write $\overline{\mathcal{M}}_n=(\mathcal{L},\|\cdot\|'_n)$ for the adelic line bundle associated to the model $(\widetilde{J}_n,\widetilde{M}_n)$.
As the constant $\exp(-2C_n/[K:\mathbb{Q}])$ converges to $1$, the sequence of the metrics of $\overline{\mathcal{M}}_n$ uniformly converges to the metric of $\hat{\mathcal{L}}$. Since the hermitian line bundles $\widetilde{\mathcal{M}}_n$ are nef,~ $\hat{\mathcal{L}}$ is nef in the sense of Section \ref{secadelic}.
\end{proof}

Next we give admissible metrics for any line bundle $\mathcal{M}$ on $X$. For this purpose, we may assume that $\mathcal{M}$ has non-zero degree. In general one obtains admissible metrics by taking tensor products. The pullback $f_{(1),\alpha}^*\mathcal{L}^{\otimes 2}$ is canonically isomorphic to $2\alpha+\omega$ on $X$. Hence, for any $\mathcal{M}$ we can find $\alpha, e, e'$ such that $\mathcal{M}^{\otimes e}\cong f_{(1),\alpha}^*\mathcal{L}^{\otimes e'}$. We write $\hat{\mathcal{M}}$ for the integrable line bundle, such that $\hat{\mathcal{M}}^{\otimes e}\cong f_{(1),\alpha}^*\hat{\mathcal{L}}^{\otimes e'}$ is an isometry, and we call $\hat{\mathcal{M}}$ admissible metrized. In particular, we obtain admissible metrics for $\alpha$ and $\omega$.

In a similar way, we obtain an admissible metric for the diagonal bundle $\Delta=\mathcal{O}_{X^2}(\Delta)$. The pullback $f_{(1,-1),\alpha}^*\mathcal{L}^{\otimes 2}$ is canonically isomorphic to $2\Delta+p_1^*\omega+p_2^*\omega$, see for example \cite[Section~4]{dJo18}. We write $\hat{\Delta}$ for the integrable bundle associated to $\Delta$, such that $$2\hat{\Delta}+p_1^*\hat{\omega}+p_2^*\hat{\omega}\cong 2f_{(1,-1),\alpha}^*\hat{\mathcal{L}}$$
is an isometry. If $s\colon X\to X^2$ denotes the embedding of the diagonal, the canonical isomorphism $s^*\Delta\cong -\omega$ yields an isometry $s^*\hat{\Delta}\cong -\hat{\omega}$.

\section{Intersection numbers and graphs}\label{secgraphs}
We discuss a combinatorial method to compute (arithmetic) intersection numbers of adelic line bundles by associating graphs. This method is based on our construction in \cite[Section~5.3]{Wil17}. 
By a graph we always mean an undirected multigraph, which can have loops.

As in the introduction, we denote the following integrable line bundles on $X^r$
\begin{align*}
\hat{\Delta}^{\alpha}_{jk}=\begin{cases} -p_j^*\hat{\omega}^{\alpha} & \text{if } j=k,\\ p_{jk}^*\hat{\Delta}^{\alpha} & \text{if }j\neq k.\end{cases}
\end{align*}
Note, that our sloppy notation ignores the dependence of $\hat{\Delta}^{\alpha}_{jk}$ on $r$.
We write
$$\mathfrak{L}_{\alpha,r}=\lbrace \hat{\Delta}^{\alpha}_{jk} \text{ on } X^r~|~ 1\le j,k\le r\rbrace$$
for the set of all these bundles.
For any tuple $(\overline{\mathcal{M}}_1,\dots,\overline{\mathcal{M}}_n)\in \mathfrak{L}_{\alpha,r}^{(n)}$ in the $n$-th symmetric power of $\mathfrak{L}_{\alpha,r}$ with $\overline{\mathcal{M}}_i=\hat{\Delta}^{\alpha}_{j_i k_i}$ we associate the graph $\Gamma_r(\overline{\mathcal{M}}_1,\dots,\overline{\mathcal{M}}_n)$ as follows:
\begin{itemize}
	\item The set of vertices of $\Gamma_r(\overline{\mathcal{M}}_1,\dots,\overline{\mathcal{M}}_n)$ is $\lbrace v_1,\dots,v_r\rbrace$.
	\item The set of edges of $\Gamma_r(\overline{\mathcal{M}}_1,\dots,\overline{\mathcal{M}}_n)$ is $\lbrace e_1,\dots,e_n\rbrace$, where the edge $e_i$ is given by $e_i=(v_{j_i},v_{k_i})$.
\end{itemize}

Vice versa, if $\Gamma$ is a graph with set of vertices $\lbrace v_1,\dots, v_r\rbrace$ and set of edges $\lbrace e_1,\dots,e_n\rbrace$, where $e_i=(v_{j_i},v_{k_i})$, we associate an $n$-tuple $(\overline{\mathcal{M}}_1,\dots,\overline{\mathcal{M}}_n)\in \mathfrak{L}_{\alpha,r}^{(n)}$, such that $\overline{\mathcal{M}}_i=\hat{\Delta}^{\alpha}_{j_i k_i}$.
We denote the intersection product of the tuple of line bundles associated to a graph $\Gamma$ by $\langle \Gamma\rangle$.
Our goal is to compute the intersection number by invariants of the graph.

We define the degree of a vertex $v_l\in V$ to be
$$\deg v_l=\#\lbrace i~|~j_i=l\rbrace+\#\lbrace i~|~k_i=l\rbrace,$$
which is just the number of edges at $v_l$, counting loops twice. Further, we write $b_0(\Gamma)$ for the number of connected components of a graph $\Gamma$.
For any admissible metrized line bundle $\hat{\mathcal{M}}$ on $X$ with $\deg\mathcal{M}=0$ we denote by $h_{NT}\left(\mathcal{M}\right)=-\langle \hat{\mathcal{M}},\hat{\mathcal{M}}\rangle$ its N\'eron--Tate height, see also \cite[(5.4)]{Zha93}. The main result of this section is the following proposition.
\begin{Pro}\label{graphintersectionnumber}
	Let $\Gamma$ be a graph with $r$ vertices and $n$ edges.
	\begin{enumerate}[(a)]
		\item
		If $r\notin\lbrace n-1,n\rbrace$, then $\langle \Gamma\rangle=0$.
		\item
		Assume $r=n$. If all vertices of $\Gamma$ have degree $2$, i.e. it is a collection of $b_0(\Gamma)$ circles, then 
		$$\langle \Gamma\rangle=(-2g)^{b_0(\Gamma)}.$$
		Otherwise, we have $\langle \Gamma\rangle=0$.
		\item
		Assume $r=n-1$ and write $x_\alpha=\alpha-\frac{\omega}{2g-2}$.
		\begin{enumerate}[(i)]
			\item If $\Gamma$ has one vertex of degree $4$ and all other vertices have degree $2$, then 
			$$\langle \Gamma\rangle=(-2g)^{b_0(\Gamma)-1}\left( \frac{g}{g-1}\hat{\omega}^2+4(g-1)h_{NT}\left(x_\alpha\right)\right).$$
			\item
			If $\Gamma$ has two vertices of degree $3$, which are connected by exactly $1$ path, and all other vertices have degree $2$, then
			$$\langle\Gamma\rangle=-4\cdot(-2g)^{b_0(\Gamma)-1}(g-1)^2h_{NT}\left(x_\alpha\right).$$
			\item
			If $\Gamma$ has two vertices of degree $3$, which are connected by $3$ different paths, and all other vertices have degree $2$, then
			$$\langle \Gamma\rangle=(-2g)^{b_0(\Gamma)-1} \left(\frac{2g+1}{2g-2}\hat{\omega}^2-\varphi(X)+6(g-1)h_{NT}\left(x_\alpha\right)\right).$$
			\item
			In any other case, we have $\langle \Gamma\rangle=0$.
		\end{enumerate}
	\end{enumerate}
\end{Pro}
We will prove the proposition by reduction steps on the graphs. For this purpose, we first need some lemmas.
\begin{Lem}\label{lemmacomponents}
	Let $\Gamma_1,\dots,\Gamma_{b_0(\Gamma)}$ be the connected components of $\Gamma$. Then we have
	$$\langle \Gamma\rangle=\prod_{j=1}^{b_0(\Gamma)} \langle \Gamma_j\rangle.$$
\end{Lem}
\begin{proof}
	This immediately follows from formula (\ref{projection}).
	 \end{proof}
\begin{Lem}\label{lemmavanishing}
	If $\Gamma$ has a vertex of degree $\le 1$, then $\langle \Gamma\rangle=0$.
\end{Lem}
\begin{proof}
	Let $(\overline{\mathcal{M}}_1,\dots,\overline{\mathcal{M}}_n)\in\mathfrak{L}_{\alpha,r}^{(n)}$ be an $n$-tuple associated to $\Gamma$ and let $v_k$ be a vertex of degree $\deg v_k\le 1$. After renaming, we may assume, that $e_1,\dots,e_{n-1}$ are not connected to $v_k$ and $e_n$ is no loop or it is also not connected to $v_k$. That means, if we factorize $\pi_r\colon X^r\to \mathrm{Spec}(K)$ by $\pi_r=\pi_{r-1}\circ p^k$, where $p^k\colon X^r\to X^{r-1}$ denotes the projection forgetting the $k$-th factor, there are integrable line bundles $\overline{\mathcal{M}}'_1,\dots,\overline{\mathcal{M}}'_{n-1}$ on $X^{r-1}$ with $\overline{\mathcal{M}}_j\cong p^{k *} \overline{\mathcal{M}}'_j$ for all $1\le j\le n-1$. Applying Equation (\ref{delgine2}) we obtain
	$$\langle \overline{\mathcal{M}}_1,\dots,\overline{\mathcal{M}}_n\rangle=c_1(\mathcal{M}_n)[p^k]\langle \overline{\mathcal{M}}'_1,\dots,\overline{\mathcal{M}}'_{n-1}\rangle.$$
	But $c_1(\mathcal{M}_n)[p^k]=0$, as $\overline{\mathcal{M}}_n$ is either isomorphic to some $p^{k *}\overline{\mathcal{M}}'_n$ or it is isomorphic to $\hat{\Delta}^{\alpha}_{jk}$ for some $j\neq k$ and
	$$c_1(\Delta^{\alpha}_{jk})[p^k]=c_1(p_{jk}^*\Delta)[p^k]-c_1(p_k^*\alpha)[p^k]=0.$$
	 \end{proof}
Let $v$ be a vertex of $\Gamma$ of degree $\deg v=2$ for which $(v,v)$ is no edge of $\Gamma$, and write $e_1=(v_1,v)$ and $e_2=(v,v_2)$ for the edges connected to $v$. We write $\Gamma/\lbrace v\rbrace$ for the graph obtained by removing $e_1, e_2$ and $v$ from $\Gamma$ and adding an edge $(v_1,v_2)$.
The next lemma shows, that the intersection number is stable under this contraction.
\begin{Lem}\label{lemmacontraction}
	Let $v$ be a vertex of $\Gamma$ of degree $\deg v=2$ for which $(v,v)$ is no edge of $\Gamma$. Then $\langle\Gamma\rangle=\langle \Gamma/\lbrace v\rbrace\rangle$.
\end{Lem}
\begin{proof}
	We choose an $n$-tuple $(\overline{\mathcal{M}}_1,\dots,\overline{\mathcal{M}}_n)\in\mathfrak{L}_{\alpha,r}^{(n)}$ associated to $\Gamma$, such that $v$ corresponds to $v_r$ in $\Gamma_r(\overline{\mathcal{M}}_{1},\dots,\overline{\mathcal{M}}_{n})$ and $\overline{\mathcal{M}}_{n-1}$ and $\overline{\mathcal{M}}_{n}$ correspond to the edges connected to $v$.
	We again factorize $\pi_r\colon X^r\to \mathrm{Spec}(K)$ by $\pi_r=\pi_{r-1}\circ p^r$ and choose $\overline{\mathcal{M}}'_j$ such that $\overline{\mathcal{M}}_j\cong p^{r *}\overline{\mathcal{M}}'_j$ for all $0\le j\le n-2$. By Equation (\ref{deligne1}) we obtain
	$$\langle\overline{\mathcal{M}}_1,\dots,\overline{\mathcal{M}}_{n}\rangle=\langle\overline{\mathcal{M}}'_0,\dots,\overline{\mathcal{M}}'_{n-2},\langle\hat{\Delta}^{\alpha}_{jr},\hat{\Delta}^{\alpha}_{rl}\rangle \rangle,$$
	for some $j,l$ different from $r$.
	Hence, it remains to compute the Deligne pairing $\langle\hat{\Delta}^{\alpha}_{jr},\hat{\Delta}^{\alpha}_{rl}\rangle$ with respect to $p^r$. As we already have seen in the proof of Lemma \ref{lemmavanishing}, $c_1(\overline{\mathcal{M}}_n)[p^r]$ vanishes. Hence, we have only to consider the non-constant part of $\hat{\Delta}^{\alpha}_{jr}$ and $\hat{\Delta}^{\alpha}_{rl}$ with respect to $p^r$:
	$$\langle\hat{\Delta}^{\alpha}_{jr},\hat{\Delta}^{\alpha}_{rl}\rangle=\left\langle p_{jr}^*\hat{\Delta}-p_r^*\hat{\alpha},p_{rl}^*\hat{\Delta}-p_r^*\hat{\alpha}\right\rangle.$$
	Let us first assume $j\neq l$.
	If we apply Equation (\ref{delignerecursive}) to the canonical section $p_{rl}^*1_{\Delta}$ of $p_{rl}^*\hat{\Delta}$, we obtain $\langle p_{jr}^*\hat{\Delta},p_{rl}^*\hat{\Delta}\rangle=p_{jl}^*\hat{\Delta}$,
	since the integral vanishes as $p_{rl}^*\hat{\Delta}$ is fibrewise admissible with respect to $p^r$.
	Similarly, we obtain
	$\langle p_r^*\hat{\alpha}, p_{jr}^*\hat{\Delta}\rangle=p_j^*\hat{\alpha}$.
	Since Deligne pairings commute with base change, we have $\langle p_r^*\hat{\alpha},p_r^*\hat{\alpha}\rangle=\pi_{r-1}^*\hat{\alpha}^2$.
	Putting everything together yields
	$$\left\langle p_{jr}^*\hat{\Delta}-p_r^*\hat{\alpha},p_{rl}^*\hat{\Delta}-p_r^*\hat{\alpha}\right\rangle=p_{jl}^*\hat{\Delta}-p_j^*\hat{\alpha}-p_l^*\hat{\alpha}+\pi_{r-1}^*\hat{\alpha}^2=\hat{\Delta}^{\alpha}_{jl}.$$
	
	Now we consider the case $j=l$. We can again apply Equation (\ref{delignerecursive}) to obtain
	$$\langle p_{jr}^*\hat{\Delta},p_{jr}^*\hat{\Delta}\rangle=p_{j}^*\langle\hat{\Delta},\hat{\Delta}\rangle=p_{j}^*(s^*\hat{\Delta})=-p_j^*\hat{\omega},$$
	where $s\colon X\to X^2$ denotes the embedding of the diagonal. As above, we compute
	$$\left\langle p_{jr}^*\hat{\Delta}-p_r^*\hat{\alpha},p_{jr}^*\hat{\Delta}-p_r^*\hat{\alpha}\right\rangle=-p_{j}^*\hat{\omega}-2p_j^*\hat{\alpha}+\pi_{r-1}^*\hat{\alpha}^2=\hat{\Delta}^{\alpha}_{jj}.$$
	This proves the lemma.
	 \end{proof}
The next lemma computes $\langle \Gamma\rangle$ for the cases, which will remain after reductions.
\begin{table}
	\caption{Intersection numbers for some graphs.}\label{tab1}
	\begin{tabular}{|c|c|c|}
		\hline
		&$\Gamma$ & $\langle\Gamma\rangle$\\ \hline
		(a)&$\xygraph{!{(0,0)}*+{\bullet}="a"  "a"-@`{p+(1,1.5),p+(1,-1.5)} "a"}$ & $-2g$\\ \hline
		(b)&$\xygraph{!{(0,0)}*+{\bullet}="a" "a"-@`{p+(-1,1.5),p+(-1,-1.5)} "a" "a"-@`{p+(1,1.5),p+(1,-1.5)} "a"}$ & $\frac{g}{g-1}\hat{\omega}^2+4(g-1)h_{NT}\left(x_\alpha\right)$ \\ \hline
		(c)&$\xygraph{!{(0,0)}*+{\bullet}="a" !{(1.0,0)}*+{\bullet}="b" "a"-@`{p+(-1,1.5),p+(-1,-1.5)} "a" "a"-"b" "b"-@`{p+(1,1.5),p+(1,-1.5)} "b" }$ & $-4(g-1)^2 h_{NT}\left(x_\alpha\right)$ \\ \hline  & & \\
		(d)&$\xygraph{!{(0,0)}*+{\bullet}="a" !{(1.5,0)}*+{\bullet}="b" "a"-@/^0.6cm/"b" "a"-"b" "a"-@/_0.6cm/"b" }$ & $\frac{2g+1}{2g-2}\hat{\omega}^2-\varphi(X)+6(g-1)h_{NT}\left(x_\alpha\right)$\\ & & \\\hline
	\end{tabular}
\end{table}
\begin{Lem}\label{lemmaremainingcases}
	The intersection numbers for the graphs $\Gamma$ in Table \ref{tab1} are given as in the table.
\end{Lem}
\begin{proof}
	\begin{enumerate}[(a)]
		\item
		We have $\langle \Gamma\rangle=\deg \left(-\omega -2\alpha\right)=-2g$.
		\item
		Since $\Gamma$ is isomorphic to $\Gamma_1(-\hat{\omega}^{\alpha},-\hat{\omega}^{\alpha})$, we obtain
		\begin{align*}
		\langle \Gamma\rangle&=\left\langle\hat{\omega}+2\hat{\alpha}-\pi^*\hat{\alpha}^2,\hat{\omega}+2\hat{\alpha}-\pi^*\hat{\alpha}^2\right\rangle,
		\end{align*}
		such that the assertion follows by a direct computation.
		\item
		As $\Gamma$ is isomorphic to $\Gamma_2(-p_1^*{\hat{\omega}}^{\alpha},\hat{\Delta}^{\alpha},-p_2^*{\hat{\omega}}^{\alpha})$, the intersection number $\langle \Gamma\rangle $ is given by
		$$\left\langle -p_1^*\hat{\omega}-2p_1^*\hat{\alpha}+\pi_2^*\hat{\alpha}^2, \hat{\Delta}-p_1^*\hat{\alpha}-p_2^*\hat{\alpha}+\pi_2^*\hat{\alpha}^2,-p_2^*\hat{\omega}-2p_2^*\hat{\alpha}+\pi_2^*\hat{\alpha}^2\right\rangle.$$
		To compute this, we proceed as in the proof of Lemma \ref{lemmacontraction}. By factorizing $\pi_2=\pi\circ p^2$ and using Equation (\ref{deligne1}) we can first multiply the last two factors, where we have $\langle\hat{\Delta},p_2^*\hat{\omega}\rangle=\hat{\omega}$ and $\langle\hat{\Delta},p_2^*\hat{\alpha}\rangle=\hat{\alpha}$ as in the proof of Lemma \ref{lemmacontraction}. This gives
		$$\left\langle-\hat{\omega}-2\hat{\alpha}+\hat{\alpha}^2,-\hat{\omega}+(2g-2)(\hat{\alpha}-\hat{\alpha}^2)+\langle\hat{\alpha},\hat{\omega}\rangle\right\rangle$$
		and by a direct computation we obtain the assertion.
		\item
		Since $\Gamma$ is isomorphic to $\Gamma_2(\hat{\Delta}^{\alpha},\hat{\Delta}^{\alpha},\hat{\Delta}^{\alpha})$, the intersection number $\langle \Gamma\rangle $ is given by
		$$\left\langle \hat{\Delta}-p_1^*\hat{\alpha}-p_2^*\hat{\alpha}+\pi_2^*\hat{\alpha}^2, \hat{\Delta}-p_1^*\hat{\alpha}-p_2^*\hat{\alpha}+\pi_2^*\hat{\alpha}^2,\hat{\Delta}-p_1^*\hat{\alpha}-p_2^*\hat{\alpha}+\pi_2^*\hat{\alpha}^2\right\rangle.$$
		Computing as before, we obtain
		$$\langle \Gamma\rangle=\langle\hat{\Delta},\hat{\Delta},\hat{\Delta}\rangle+6(g-1)h_{NT}\left(\alpha-\frac{\omega}{2g-2}\right)+\frac{3\hat{\omega}^2}{2g-2}.$$
		It was shown in \cite[Proposition~7.1]{dJo18}, that $\langle\hat{\Delta},\hat{\Delta},\hat{\Delta}\rangle=\hat{\omega}^2-\varphi(X)$. Hence, we obtain the assertion.
	\end{enumerate}
	 \end{proof}
Now we can prove Proposition \ref{graphintersectionnumber}.
\begin{proof}[Proof of Proposition \ref{graphintersectionnumber}]
	\begin{enumerate}[(a)]
		\item
		This directly follows, as $\langle\overline{\mathcal{M}}_1,\dots,\overline{\mathcal{M}}_n|X^r\rangle$ vanishes if $r\notin\lbrace n-1,n\rbrace$.
		\item
		By Lemma \ref{lemmacontraction} a circle has the same intersection number as a loop, which has intersection number $-2g$ by Lemma \ref{lemmaremainingcases}. Hence, a collection of $b_0$ circles has intersection number $(-2g)^{b_0}$ by Lemma \ref{lemmacomponents}.
		If there is a vertex of degree $\neq 2$, then there has to be a vertex of degree $\le 1$, since $n=r$. Then the intersection number vanishes by Lemma \ref{lemmavanishing}.
		\item
		If we are in case (iv), there has to exist a vertex of degree $\le 1$. Hence, the intersection number $\langle \Gamma\rangle$ vanishes. Therefore, we may assume that there is no vertex of degree $\le 1$.
		If we decompose $\Gamma$ into its connected components and contract all degree $2$ vertices $v$ for which $(v,v)$ is no edge of $\Gamma$, we obtain $b_0(\Gamma)-1$ loops and one graph of the form (b), (c) or (d) in Table \ref{tab1}. By Lemmas \ref{lemmacomponents}, \ref{lemmacontraction} and \ref{lemmaremainingcases} we obtain the intersection number as in the proposition.
	\end{enumerate}
	 \end{proof}
\section{Proof of Theorem \ref{mainthm}}\label{secproofmain}
We will prove Theorem \ref{mainthm} in this section.
Let  $\overline{\mathcal{M}}_1,\dots,\overline{\mathcal{M}}_{r+1}$ be adelic $\mathbb{Q}$-line bundles on $X^r$ of the form
$$\overline{\mathcal{M}}_l=\tfrac{1}{2}\sum_{j,k=1}^r t_{l,j,k}\hat{\Delta}^{\alpha}_{jk}=\tfrac{1}{2}\sum_{j=1}^r t_{l,j,j}\hat{\Delta}^{\alpha}_{jj}+\sum_{j<k}^r t_{l,j,k}\hat{\Delta}^{\alpha}_{jk},$$
where $t_{l,j,k}\in\mathbb{Q}$ are rational numbers satisfying $t_{l,j,k}=t_{l,k,j}$ for all $l,j,k$.

We first prove (a). By multi-linearity we can expand $\langle \mathcal{M}_1,\dots,\mathcal{M}_{r}\rangle$ as a linear combination of intersection numbers of graphs as in the previous section. 
We have to count with coefficients the graphs consisting only of circles. For this purpose, let $\Pi_r$ be the set of all partitions of $\lbrace 1,\dots, r\rbrace$.
To any graph $\Gamma$ with vertex set isomorphic to $\lbrace 1,\dots, r\rbrace$ we associate the partition $\pi_\Gamma\in\Pi_r$ induced by the connected components of $\Gamma$. If $\Gamma$ only consists of circles, its intersection number is given by $\langle\Gamma\rangle=(-2g)^{|\pi_\Gamma|}$ and hence, it only depends on $\pi_\Gamma$. Fix some $\pi\in\Pi_r$ and $B\in\pi$. To build a circle of the elements in $B$, we choose an isomorphism $\sigma\colon\mathbb{Z}/|B|\xrightarrow{\sim}B$. Of course, neither the starting point $\sigma(0)$ nor the direction is a datum of the circle, such that by symmetry we get every circle $2|B|$-times if $|B|\ge 3$, and $|B|$-times else. 

Thus, the coefficient corresponding to the partition $\pi$ in the expansion of the intersection number $\langle\overline{M}_1,\dots,\overline{\mathcal{M}}_{r}\rangle$ is given by
\begin{align}\label{partition}
\sum_{\tau\in\mathcal{S}_r} \prod_{B\in\pi}\frac{1}{2|B|}\sum_{\sigma\colon\mathbb{Z}/|B|\xrightarrow{\sim}B}\prod_{j=0}^{|B|-1}t_{\tau(\sigma(j)),\sigma(j),\sigma(j+1)}.
\end{align}
Note that we have to sum over the symmetric group $\mathcal{S}_r$ of $\lbrace 1,\dots, r\rbrace$, since fixing the graph $\Gamma$ determines the corresponding $n$-tuple of line bundles only up to order. By the factor $\frac{1}{2}$ in the definition of $\overline{\mathcal{M}}_l$, we get the $2$ in the denominator also if $|B|=1$. If $|B|=2$ we obtain the $2$ in the denominator, since in this case the two line bundles corresponding to the circle associated to $B$ are equal, such that we have to divide by $2$ after summing over all permutation $\tau\in\mathcal{S}_{r}$ of the line bundles.

If we multiply (\ref{partition}) with the intersection number $(-2g)^{|\pi|}$ of the corresponding graphs and sum over all partitions $\pi\in\Pi_r$, we obtain the formula in the theorem. This proves (a).

To prove (b), we may again expand $\langle \overline{\mathcal{M}}_1,\dots,\overline{\mathcal{M}}_{r+1}\rangle$ as a linear combination of intersection numbers of graphs. Now we have to count with coefficients the graphs of the form as in (i), (ii) and (iii) of Proposition \ref{graphintersectionnumber} (c). We consider the three different types separately for a fixed partition $\pi\in\Pi_r$.
\begin{enumerate}[(i)]
	\item
	We have to build circles of the elements of $B$ for all $B\in \pi$ except for one. Denote this one by $B'$. We have to build a connected graph with a vertex of degree $4$ and all other vertices of degree $2$ of the elements of $B'$. We do this by fixing a $k\in B'$, which will be the vertex of degree $4$, and choosing an isomorphism $\sigma \colon \mathbb{Z}/|B'|\xrightarrow{\sim}B'$.
	The graph is obtained by connecting $\sigma(j)$ and $\sigma(j+1)$ by an edge for every $1\le j\le |B'|-1$ and connecting $\sigma(0)$ and $k$ by an edge and $\sigma(1)$ and $k$ by another edge. If $|B'|=11$ and $k=\sigma(6)$, this looks as follows:
	$$\xygraph{!{(0,0)}*+!(.15,0){k \bullet}="f" !{(-1,1)}*+!(-.3,0){\bullet\sigma(1)}="a" !{(-2,1)}*+!(.3,0){\sigma(2) \bullet}="b" !{(-3,0)}*+!(.3,0){\sigma(3)\bullet}="c" !{(-2,-1)}*+!(.3,0){\sigma(4) \bullet}="d" !{(-1,-1)}*+!(-.3,0){\bullet\sigma(5)}="e" !{(1,1)}*+!(.3,0){\sigma(7)\bullet}="g" !{(2,1)}*+!(-.3,0){\bullet\sigma(8)}="h" !{(3,0)}*+!(-.3,0){\bullet\sigma(9)}="i" !{(2,-1)}*+!(-.3,0){\bullet\sigma(10)}="j" !{(1,-1)}*+!(.3,0){\sigma(0)\bullet}="k" "a"-"b" "b"-"c" "c"-"d" "d"-"e" "e"-"f" "f"-"g" "g"-"h" "h"-"i" "i"-"j" "j"-"k" "a"-@{--}"f" "f"-@{--}"k"}.$$
	
	The graph does not care about the direction and we can flip the direction of the left side respectively the right side. Hence, we obtain every graph $8$ times if both circles consists of at least $3$ vertices. If we have one circle with at least $3$ vertices and the other circle has less than $3$ vertices, we obtain the graph $4$ times. If both circles have less than $3$ vertices but we have at least $2$ vertices in all, we obtain the graph $2$ times. Finally, we obtain the graph consisting of two loops exactly once.
	Similarly to the proof of (a), the coefficient corresponding to the partition $\pi$ in the expansion of the intersection number $\langle\overline{\mathcal{M}}_1,\dots,\overline{\mathcal{M}}_{r+1}\rangle$ is given by
	\begin{align*}
	&\sum_{\tau\in\mathcal{S}_{r+1}}\sum_{B'\in\pi}\left(\prod_{B\in \pi\setminus\lbrace B'\rbrace}\frac{1}{2|B|}\sum_{\sigma\colon \mathbb{Z}/|B|\xrightarrow{\sim}B}\prod_{j=0}^{|B|-1}t_{\tau(\sigma(j)),\sigma(j),\sigma(j+1)}\right)\times\\
	&\frac{1}{8}\sum_{k\in B'}\sum_{\sigma\colon \mathbb{Z}/|B'|\xrightarrow{\sim} B'}t_{\tau(\sigma(0)),\sigma(0),k}t_{\tau(r+1),\sigma(1),k}\prod_{j=1}^{|B'|-1}t_{\tau(\sigma(j)),\sigma(j),\sigma(j+1)}.
	\end{align*}
	We have to explain, why we uniformly get the factor $\frac{1}{8}$. By the factor $\frac{1}{2}$ in the definition of $\overline{\mathcal{M}}$, we obtain an additional factor $2$ in the denominator for every loop in the graph associated to $B'$. Further, we have to divide by $2$ whenever the tuple of line bundles corresponding to the graph associated to $B'$ contains two equal line bundles, since we summed over all permutations $\tau\in\mathcal{S}_{r+1}$. Going through the cases mentioned above, one checks that we always have to divide by $8$.
	
	If we multiply with the intersection number of the graph, which only depends on the partition associated to the graph, and sum over all partitions, we obtain by Proposition \ref{graphintersectionnumber}
	\begin{align}\label{part1}
	\frac{c_1}{8}\left(\frac{g}{g-1}\hat{\omega}^2+4(g-1) h_{NT}\left(x_\alpha\right)\right)
	\end{align}
	with $c_1$ as in Theorem \ref{mainthm}.
	\item
	The other two cases are similar to (i) and we try to give only the differences. We again fix a $B'\in \pi$ and we build a graph with two vertices of degree $3$, which are connected by exactly $1$ path, and all other vertices have degree $2$. We choose again an isomorphism $\sigma\colon \mathbb{Z}/|B'|\xrightarrow{\sim}B'$. Further, we fix two integers $1\le j<k\le|B'|$.
	The graph is obtained by connecting $\sigma(l)$ and $\sigma(l+1)$ for all $1\le l\le |B'|-1$ and connecting $\sigma(0)$ to $\sigma(k)$ and $\sigma(1)$ to $\sigma(j)$. If $|B'|=11$, $j=4$ and $k=8$, this looks as follows:
	
	\begin{align*}
	\xygraph{!{(3,1)}*+!(-.3,0){\bullet \sigma(0)}="k" !{(-3,1)}*+!(-.3,0){\bullet\sigma(1)}="a" !{(-3,0)}*+!(.3,0){\sigma(2) \bullet}="b" !{(-3,-1)}*+!(-.3,0){\bullet \sigma(3)}="c" !{(-2,0)}*+!(.3,0){\sigma(4) \bullet}="d" !{(-1,0)}{\bullet\sigma(5)}="e" !{(0,0)}{\bullet\sigma(6)}="f" !{(1,0)}{\bullet\sigma(7)}="g" !{(2,0)}*+!(-.3,0){\bullet \sigma(8)}="h" !{(3,-1)}*+!(-.3,0){\bullet\sigma(9)}="i" !{(3,0)}*+!(-.3,0){\bullet\sigma(10)}="j" "a"-"b" "b"-"c" "c"-"d" "d"-"e" "e"-"f" "f"-"g" "g"-"h" "h"-"i" "i"-"j" "j"-"k" "a"-@{--}"d" "h"-@{--}"k"}.
	\end{align*}
	We again obtain every graph $8$ times, $4$ times or twice, depending on the cardinality of the vertices in the two circles. But as above by the factor $\frac{1}{2}$ in the definition of $\overline{\mathcal{M}}_l$ and after summing over all permutations $\tau\in\mathcal{S}_{r+1}$ of the line bundles in the intersection product of the graph, we can uniformly write
	\begin{align*}
	&\frac{1}{8}\sum_{\tau\in\mathcal{S}_{r+1}}\sum_{B'\in\pi}\left(\prod_{B\in \pi\setminus\lbrace B'\rbrace}\frac{1}{2|B|}\sum_{\sigma\colon \mathbb{Z}/|B|\xrightarrow{\sim}B}\prod_{j=0}^{|B|-1}t_{\tau(\sigma(j)),\sigma(j),\sigma(j+1)}\right)\times\\
	&\sum_{1\le j<k\le |B'|}\sum_{\sigma\colon \mathbb{Z}/|B'|\xrightarrow{\sim} B'}t_{\tau(\sigma(0)),\sigma(0),\sigma(k)}t_{\tau(r+1),\sigma(1),\sigma(j)}\prod_{j=1}^{|B'|-1}t_{\tau(\sigma(j)),\sigma(j),\sigma(j+1)}
	\end{align*}
	for the coefficient corresponding to the partition $\pi$ in the expansion of the intersection number $\langle\overline{\mathcal{M}}_1,\dots,\overline{\mathcal{M}}_{r+1}\rangle$.
	Multiplying with the intersection number of the corresponding graphs and summing over all partitions yields
	\begin{align}\label{part2}
	-\frac{c_2(g-1)^2}{2} h_{NT}\left(x_\alpha\right),
	\end{align}
	with $c_2$ as in Theorem \ref{mainthm}.
	\item
	In this case we have to build a graph of the elements of $B'$, which has two vertices of degree $3$, connected by $3$ different paths, and all other vertices have degree $2$ . We again fix $1\le j<k\le|B'|$ and an isomorphism $\sigma\colon \mathbb{Z}/|B'|\xrightarrow{\sim}B'$. The graph is obtained by connecting $\sigma(l)$ and $\sigma(l+1)$ for every $1\le l\le |B'|-1$ and connecting $\sigma(0)$ to $\sigma(j)$ and $\sigma(1)$ to $\sigma(k)$. For $|B'|=11$, $j=4$ and $k=8$, this looks as follows:
	$$\xygraph{!{(1,-1)}*+!(-.3,0){\bullet \sigma(0)}="k" !{(-1,1)}*+!(.3,0){\sigma(1)\bullet}="a" !{(0,1)}{\bullet\sigma(2) }="b" !{(1,1)}*+!(-.3,0){\bullet \sigma(3)}="c" !{(2,0)}*+!(-.3,0){\bullet\sigma(4)}="d" !{(1,0)}{\bullet\sigma(5)}="e" !{(0,0)}{\bullet\sigma(6)}="f" !{(-1,0)}{\bullet\sigma(7)}="g" !{(-2,0)}*+!(.3,0){ \sigma(8)\bullet}="h" !{(-1,-1)}*+!(.3,0){\sigma(9)\bullet}="i" !{(0,-1)}{\bullet\sigma(10)}="j" "a"-"b" "b"-"c" "c"-"d" "d"-"e" "e"-"f" "f"-"g" "g"-"h" "h"-"i" "i"-"j" "j"-"k" "a"-@{--}"h" "d"-@{--}"k"}.$$
	
	If at least two of the three paths has at least one vertex in between, we obtain the graph $12$ times, since we can reverse the direction and we can interchange the three paths. Otherwise, there are line bundles in the intersection product of the graph occurring twice or three times. Hence, after summing over all permutations $\tau\in\mathcal{S}_{r+1}$, we can uniformly write
	\begin{align*}
	&\frac{1}{12}\sum_{\tau\in\mathcal{S}_{r+1}}\sum_{B'\in\pi}\left(\prod_{B\in \pi\setminus\lbrace B'\rbrace}\frac{1}{2|B|}\sum_{\sigma\colon \mathbb{Z}/|B|\xrightarrow{\sim}B}\prod_{j=0}^{|B|-1}t_{\tau(\sigma(j)),\sigma(j),\sigma(j+1)}\right)\times\\
	&\sum_{1\le j<k\le |B'|}\sum_{\sigma\colon \mathbb{Z}/|B'|\xrightarrow{\sim} B'}t_{\tau(\sigma(0)),\sigma(0),\sigma(j)}t_{\tau(r+1),\sigma(1),\sigma(k)}\prod_{j=1}^{|B'|-1}t_{\tau(\sigma(j)),\sigma(j),\sigma(j+1)}
	\end{align*}
	for the coefficient corresponding to the partition $\pi$ in the expansion of the intersection number $\langle\overline{\mathcal{M}}_1,\dots,\overline{\mathcal{M}}_{r+1}\rangle$.
	Multiplying with the intersection number of the graph and summing over all partitions yields:
	\begin{align}\label{part3}
	\frac{c_3}{12} \left(\frac{2g+1}{2g-2}\hat{\omega}^2-\varphi(X)+6(g-1)h_{NT}\left(x_\alpha\right)\right),
	\end{align}
	with $c_3$ as in Theorem \ref{mainthm}.
\end{enumerate}
Now $\langle \overline{\mathcal{M}}_1,\dots,\overline{\mathcal{M}}_{r+1}\rangle$ is given by the sum of (\ref{part1}), (\ref{part2}) and (\ref{part3}). This proves Theorem \ref{mainthm}.
\section{Proof of Theorem \ref{neron-tate-height}}\label{secnerontate}
In this section we will prove Theorem \ref{neron-tate-height} by computing the intersection numbers $\langle (f_{m,\alpha}^*\mathcal{L})^r\rangle$ and $\langle (f_{m,\alpha}^*\hat{\mathcal{L}})^{r+1}\rangle$ as a special case of Theorem \ref{mainthm}. Let us first express $f_{m,\alpha}^*\hat{\mathcal{L}}$ by the integrable line bundles $\hat{\Delta}^{\alpha}_{jk}$.
\begin{Lem}\label{lemmapullback}
	It holds $f_{m,\alpha}^*\hat{\mathcal{L}}=-\frac{1}{2}\sum_{j,k=1}^r m_j m_k \hat{\Delta}^{\alpha}_{jk}$.
\end{Lem}
\begin{proof}
	By \cite[Equation (6.1)]{dJo18} we have
	$$f_{m,\alpha}^*\hat{\mathcal{L}}=\frac{1}{2}\sum_{j=1}^r m_j^2p_j^*\hat{\omega}-\sum^r_{j<k}m_jm_kp_{jk}^*\hat{\Delta}+d\sum_{j=1}^rm_jp_j^*\hat{\alpha}-\frac{d^2}{2}\pi^*\hat{\alpha}^2$$
	with $d=\sum_{j=1}^r m_j$.
	The assertion directly follows by the definition of $\hat{\Delta}_{jk}^{\alpha}$.
	 \end{proof}
For any $x\in\mathbb{R}$ and $n\in \mathbb{Z}$ we write $(x)_n=\prod_{k=0}^{n-1}(x-k)$ for the $n$-th falling factorial of $x$. Note, that for $n\le 0$ we get the empty product $(x)_n=1$.
Theorem \ref{neron-tate-height} follows directly from the following special case of Theorem \ref{mainthm} and the projection formula (\ref{projectionformula}).
\begin{Thm}\label{intersectionpullback}
	We have $\langle (f_{m,\alpha}^*\mathcal{L})^r\rangle=r!(g)_{r} \prod_{j=1}^r m_j^2$. Further, it holds
	$$\langle (f_{m,\alpha}^*\hat{\mathcal{L}})^{r+1}\rangle=\frac{(r+1)!\prod_{j=1}^r m_j^2}{24}\left(a'\hat{\omega}^2+b'\varphi(X)+c'h_{NT}(x_\alpha)\right)$$
	with
	$$a'=\frac{3g(g-2)_{r-1}}{g-1}\sum_{j=1}^r m_j^2-\frac{2(2g+1)(g-3)_{r-2}}{g-1}\sum^r_{j< k}m_j m_k,$$
	$$b'=4(g-3)_{r-2}\sum_{j< k}^rm_j m_k\quad\text{and}\quad c'=12(g-1)_r\left(\sum_{j=1}^rm_j\right)^2.$$
\end{Thm}
\begin{proof}
	By Theorem \ref{mainthm}, we have
	\begin{align}\label{selfproduct_of_thetabundle}
	\langle (f_{m,\alpha}^*\mathcal{L})^r\rangle=(-1)^rr!\sum_{\pi\in\Pi_r}(-g)^{|\pi|}\prod_{B\in\pi}\frac{1}{|B|}\sum_{\sigma\colon\mathbb{Z}/|B|\xrightarrow{\sim}B}\prod_{j=0}^{|B|-1}m_{\sigma(j)} m_{\sigma(j+1)}.
	\end{align}
	This is a polynomial in $g$ of degree $r$. Note, that (\ref{selfproduct_of_thetabundle}) also holds for $g=1$, since the assumption $g\ge 2$ was only needed for the proof of part (b) in Theorem \ref{mainthm}. Hence, we have $\langle (f_{m,\alpha}^*\mathcal{L})^r\rangle=0$ for $1\le g\le r-1$ by (\ref{vanishingbydimension}), such that the polynomial vanishes for $1\le g\le r-1$. Since the polynomial is also divisible by $g$, it has to be a multiple of $(g)_r$. Considering the summand for $\pi=\lbrace\lbrace1\rbrace,\lbrace 2\rbrace, \dots\lbrace r\rbrace\rbrace$, we obtain that its leading coefficient is $r!\prod_{j=1}^r m_j^2$. Therefore, we get $\langle (f_{m,\alpha}^*\mathcal{L})^r\rangle=r!(g)_{r} \prod_{j=1}^r m_j^2$ as desired.
	
	Now we consider $\langle (f_{m,\alpha}^*\hat{\mathcal{L}})^{r+1}\rangle$.
	Let us compute the numbers $c_1, c_2$ and $c_3$ in Theorem \ref{mainthm} in this particular case.
	First, we obtain that $(-1)^{r+1}c_1$ is equal to
	\begin{align*}
	(r+1)!\sum_{\pi\in \Pi_{r}}(-g)^{|\pi|-1}\sum_{\gf{B'\in\pi}{k\in B'}}|B'|m_k^2\prod_{B\in\pi}\frac{1}{|B|}\sum_{\sigma\colon \mathbb{Z}/|B|\xrightarrow{\sim}B}\prod_{j=0}^{|B|-1}m_{\sigma(j)}m_{\sigma(j+1)}.
	\end{align*}
	We consider $c_1$ as a polynomial in $g$ of degree $r-1$.
	To compute this polynomial we may assume, that $X$ is hyperelliptic and $\alpha=\frac{\omega}{2g-2}$. Then we have $\frac{2g+1}{2g-2}\hat{\omega}^2=\varphi$ and $h_{NT}(x_\alpha)=0$. Hence, $\langle (f_{m,\alpha})^*\hat{\mathcal{L}}^{r+1}\rangle=\frac{gc_1}{8(g-1)}\hat{\omega}^2$. But $\langle (f_{m,\alpha}^*\hat{\mathcal{L}})^{r+1}\rangle$ vanishes for $2\le g\le r-1$ by (\ref{vanishingbydimension}). For $g=r$ we also have $\langle (f_{m,\alpha}^*\hat{\mathcal{L}})^{r+1}\rangle=\langle \hat{\mathcal{L}}^{r+1}\rangle=0$, as the N\'eron--Tate height $h'_{\mathcal{L}}(J)$ vanishes.
	Thus, the polynomial is a multiple of $(g-2)_{r-1}$.
	Considering the summand for $\pi=\lbrace\lbrace1\rbrace,\lbrace 2\rbrace, \dots\lbrace r\rbrace\rbrace$, we obtain that the leading coefficient of $c_1$ is $(r+1)!\sum_{k=1}^rm_k^2\prod_{j=1}^rm_j^2$. Therefore, we get
	$$c_1=(r+1)!(g-2)_{r-1}\sum_{k=1}^rm_k^2\prod_{j=1}^rm_j^2.$$
	
	Further, $(-1)^{r+1}c_3$ is given by
	\begin{align*}
	(r+1)!\sum_{\pi\in \Pi_{r}}(-g)^{|\pi|-1}\sum_{B'\in \pi}\sum_{\gf{j,k\in B'}{j<k}}|B'|m_jm_k\prod_{B\in\pi}\frac{1}{|B|}\sum_{\sigma\colon \mathbb{Z}/|B|\xrightarrow{\sim}B}\prod_{j=0}^{|B|-1}m_{\sigma(j)}m_{\sigma(j+1)}.
	\end{align*}
	Since the summand for $\pi=\lbrace\lbrace1\rbrace,\lbrace 2\rbrace, \dots\lbrace r\rbrace\rbrace$ vanishes, $c_3$ is a polynomial in $g$ of degree $r-2$. We again assume $\alpha=\frac{\omega}{2g-2}$. If we assume $g\ge 3$, we can choose $X$, such that $\frac{2g+1}{2g-2}\hat{\omega}^2\neq\varphi$. There always exists such a curve, as $\frac{2g+1}{2g-2}\hat{\omega}^2-\varphi$ satisfies a Northcott property on every projective subvariety of the coarse moduli space of curves $\mathcal{M}_g$ by \cite[Theorem 1.3.5]{Zha10} and $\mathcal{M}_g$ contains projective curves for every $g\ge 3$, see for example \cite[Lecture 1, Section 3.]{Har87}.
	Then we have
	$$\langle (f_{m,\alpha}^*\hat{\mathcal{L}})^{r+1}\rangle=\frac{gc_1}{8(g-1)}\hat{\omega}^2+\frac{c_3}{12}\left(\frac{2g+1}{2g-2}\hat{\omega}^2- \varphi(X)\right).$$
	As $\langle (f_{m,\alpha}^*\hat{\mathcal{L}})^{r+1}\rangle$ and $c_1$ vanish for $3\le g\le r$, also $c_3$ has to vanish for $3\le g\le r$. Hence, $c_3$ is a multiple of $(g-3)_{r-2}$. Considering the summands of $c_3$ for 
	$$\pi_{jk}=\lbrace\lbrace i\rbrace~|~i\notin\lbrace j,k\rbrace\rbrace\cup\lbrace\lbrace j,k\rbrace\rbrace$$
	for all pairs $1\le j<k\le r$, we obtain $-2(r+1)!\sum_{j<k}^rm_jm_k\prod_{j=1}^r m_j^2$ for the leading coefficient of $c_3$. Hence, we obtain
	$$c_3=-2(r+1)!(g-3)_{r-2}\sum_{j<k}^rm_jm_k\prod_{j=1}^r m_j^2.$$
	It can be directly checked, that we have $c_2=c_3$ in this particular situation. Now the theorem follows by putting the values for $c_1$, $c_2$ and $c_3$ into the formula in Theorem \ref{mainthm}.
	 \end{proof}
\section{Lower bounds for the self-intersection number}
We prove Theorem \ref{lowerbound} by applying the arithmetic Hodge index theorem for adelic line bundles by Yuan--Zhang \cite[Theorem~3.2]{YZ17} to Theorem \ref{mainthm}. 
For any $t_{jk}\in \mathbb{Q}$ for $1\le j,k\le r$ we define on $X^r$ the adelic line bundle $\overline{\mathcal{M}}=\sum_{j,k=1}^r m_j m_k t_{jk}\hat{\Delta}_{jk}^{\alpha}$. Our goal is to apply the arithmetic Hodge index theorem to obtain the non-positivity of $\langle (f^*_{m,\alpha}\hat{\mathcal{L}})^{r-1},\overline{\mathcal{M}}^2\rangle$. Hence, we have to check, whether the assumptions of the theorem are satisfied.
We recall from \cite[Theorem~3.2]{YZ17}, that $\langle (f^*_{m,\alpha}\hat{\mathcal{L}})^{r-1},\overline{\mathcal{M}}^2\rangle\le 0$ if
\begin{itemize}
	\item $f^*_{m,\alpha}\hat{\mathcal{L}}$ is a nef adelic line bundle,
	\item $f^*_{m,\alpha}\mathcal{L}$ is big and
	\item $\langle (f^*_{m,\alpha}\mathcal{L})^{r-1},\mathcal{M}\rangle=0$.
\end{itemize}

As $\hat{\mathcal{L}}$ is a nef adelic line bundle by Lemma \ref{lemnef}, $f_{m,\alpha}^*\hat{\mathcal{L}}$ is nef, too. 
Since $f^*_{m,\alpha}\mathcal{L}$ is nef, it is big if and only if $\mathrm{vol}(f^*_{m,\alpha}\mathcal{L})=\langle (f^*_{m,\alpha}\mathcal{L})^r\rangle>0$. By Theorem \ref{intersectionpullback}, we have $\langle (f^*_{m,\alpha}\mathcal{L})^r\rangle=r!(g)_r\prod_{j=1}^rm_j^2$, which is positive, if and only if $r\le g$.
For the third assumption we give the following lemma.
\begin{Lem}\label{vanishingintersection}
	It holds
	$$\langle (f^*_{m,\alpha}\mathcal{L})^{r-1},\mathcal{M}\rangle=-2(r-1)!(g)_{r-1}\prod_{j=1}^r m_j^2\left(g\sum_{j=1}^r t_{jj}-\sum_{j\neq k}^r t_{jk}\right).$$
\end{Lem}
\begin{proof}
	Let us first denote the following vectors in $\left(\mathbb{Z}\setminus\lbrace 0\rbrace \right)^{r-1}$
	$$m(i)=(m_1,\dots,m_{i-1},m_{i+1},\dots,m_r) \text{ and}$$
	$$m(j,k)=(m_1,\dots,m_{j-1},m_j+m_k,m_{j+1},\dots,m_{k-1},m_{k+1},\dots,m_r)$$
	for $j<k$.
	To compute $\langle (f^*_{m,\alpha}\mathcal{L})^{r-1},\mathcal{M}\rangle$, we decompose
	\begin{align}\label{decomposition}
	\langle (f^*_{m,\alpha}\mathcal{L})^{r-1},\mathcal{M}\rangle=\sum_{j,k=1}^r m_j m_k t_{jk}\langle (f^*_{m,\alpha}\mathcal{L})^{r-1},\Delta_{jk}^\alpha\rangle.
	\end{align}
	Let us compute the intersection products $\langle (f^*_{m,\alpha}\mathcal{L})^{r-1},\Delta_{jk}^\alpha\rangle$. We first consider the case $j\neq k$. By symmetry we may assume $j<k$. By definition we have $\Delta^\alpha_{jk}=p_{jk}^*\Delta-p_j^*\alpha-p_k^*\alpha$. Further, we define the maps
	$$s_{jk}\colon X^{r-1}\to X^r,\quad (x_1,\dots,x_{r-1})\mapsto(x_1,\dots,x_{k-1},x_j,x_{k+1},\dots,x_{r-1}).$$
	A direct computation gives $f_{m,\alpha}\circ s_{jk}=f_{m(j,k),\alpha}$. As $p_{jk}^*\Delta=\mathcal{O}_{X^r}(s_{jk})$, we can conclude from Theorem \ref{intersectionpullback} 
	$$\langle (f^*_{m,\alpha}\mathcal{L})^{r-1},p_{jk}^*\Delta\rangle=\langle (s_{jk}^*f^*_{m,\alpha}\mathcal{L})^{r-1}\rangle=(r-1)!(g)_{r-1}(m_j+m_k)^2\prod_{\gf{l=1}{l\notin\lbrace j,k\rbrace}}^r m_l^2.$$
	As it holds $c_1(f_{m,\alpha}^*\mathcal{L})^{r-1}[p^j]=c_1(f_{m(j),\alpha}^*\mathcal{L})^{r-1}$, we obtain by Equation (\ref{delgine2}) and Theorem \ref{intersectionpullback}
	$$\langle (f^*_{m,\alpha}\mathcal{L})^{r-1},p_{j}^*\alpha\rangle=\langle (f_{m(j),\alpha}^*\mathcal{L})^{r-1}\rangle\cdot \deg\alpha=(r-1)!(g)_{r-1}\prod_{\gf{l=1}{l\neq j}}^r m_l^2.$$
	Putting this together, we get
	$$\langle (f^*_{m,\alpha}\mathcal{L})^{r-1},\Delta^{\alpha}_{jk}\rangle=2(r-1)!(g)_{r-1}m_j m_k\prod_{\gf{l=1}{l\notin\lbrace j,k\rbrace}}^r m_l^2.$$
	As above, we obtain by $c_1(f_{m,\alpha}^*\mathcal{L})^{r-1}[p^j]=c_1(f_{m(j),\alpha}^*\mathcal{L})^{r-1}$, Equation (\ref{delgine2}) and Theorem \ref{intersectionpullback}
	$$\langle (f^*_{m,\alpha}\mathcal{L})^{r-1},p_{j}^*\Delta^\alpha_{jj}\rangle=\langle (f_{m(j),\alpha}^*\mathcal{L})^{r-1}\rangle\cdot (-2g)=(-2g)(r-1)!(g)_{r-1}\prod_{\gf{l=1}{l\neq j}}^r m_l^2.$$
	Now the lemma follows by applying these computations to the decomposition in Equation (\ref{decomposition}).
	 \end{proof}
Now we can give the proof of Theorem \ref{generallowerbound}.
\begin{proof}[Proof of Theorem \ref{generallowerbound}]
	By Lemma \ref{vanishingintersection} the assumption $g\sum_{j=1}^r t_{jj}=\sum_{j\neq k}^r t_{jk}$ implies $\langle (f^*_{m,\alpha}\mathcal{L})^{r-1},\mathcal{M}\rangle=0$.
	Thus, for $r\le g$ the theorem directly follows by the arguments above. For $r\ge g+3$ the arithmetic intersection number $\langle(f_{m,\alpha}^*\hat{\mathcal{L}})^{r-1},\overline{\mathcal{M}}^2\rangle$ vanishes by (\ref{vanishingbydimension}).
	For $r=g+2$ we also obtain $\langle(f_{m,\alpha}^*\hat{\mathcal{L}})^{r-1},\overline{\mathcal{M}}^2\rangle=0$, as $\langle \hat{\mathcal{L}}^{r-1}\rangle$ is a multiple of $h'_{\mathcal{L}}(J)=0$.
	
	For $r=g+1$ we have to use an additional argument, since $f_{m,\alpha}^*\mathcal{L}$ is no longer big. Consider the adelic line bundle
	$$\overline{\mathcal{A}}=-\frac{1}{2}\sum_{j=1}^{g+1}m_j^2\hat{\Delta}^{\alpha}_{jj}=\sum_{j=1}^{g+1} pr_j^* (f_{(m_j),\alpha}^*\hat{\mathcal{L}}).$$
	This adelic line bundle is nef, since it is a positive linear combination of pullbacks of $\hat{\mathcal{L}}$. Moreover, $\mathcal{A}$ is big, because we have 
	$$\mathrm{vol}(\mathcal{A})=\langle \mathcal{A}^{g+1}\rangle=(g+1)! g^{g+1}\prod_{j=1}^{g+1} m_j^2>0$$
	by Theorem \ref{mainthm}.
	Hence, for any rational $\epsilon>0$, the adelic $\mathbb{Q}$-line bundle $(f_{m,\alpha}^*\hat{\mathcal{L}}+\epsilon\overline{\mathcal{A}})$ is nef and $(f_{m,\alpha}^*\mathcal{L}+\epsilon\mathcal{A})$ is big.
	There exists a unique rational number $c(\epsilon)$ depending continuously on $\epsilon$, such that
	$$\langle (f^*_{m,\alpha}\mathcal{L}+\epsilon\mathcal{A})^g,(\mathcal{M}+c(\epsilon)\mathcal{A})\rangle=0.$$
	Now the arithmetic Hodge index theorem \cite[Theorem~3.2]{YZ17} implies, that
	$$\langle (f^*_{m,\alpha}\hat{\mathcal{L}}+\epsilon\overline{\mathcal{A}})^{g},(\overline{\mathcal{M}}+c(\epsilon)\overline{\mathcal{A}})^2\rangle\le 0.$$
	By continuity, we also obtain $\langle (f^*_{m,\alpha}\hat{\mathcal{L}})^g,(\overline{\mathcal{M}}+c(0)\overline{\mathcal{A}})^2\rangle\le 0$, where $c(0)$ is defined as the limit $\lim_{\epsilon\to 0} c(\epsilon)$. Thus, we only have to show $c(0)=0$. For this purpose, we consider
	$$0=\langle (f^*_{m,\alpha}\mathcal{L})^g,(\mathcal{M}+c(0)\mathcal{A})\rangle=c(0)\langle (f^*_{m,\alpha}\mathcal{L})^g,\mathcal{A}\rangle.$$
	But the same argument as in the proof of Lemma \ref{vanishingintersection} gives
	$$\langle (f^*_{m,\alpha}\mathcal{L})^g,\mathcal{A}\rangle=g\sum_{j=1}^{g+1} m_j^2 \langle (f_{m(j),\alpha}^*\mathcal{L})^g\rangle=g \cdot (g+1)!\cdot g!\prod_{j=1}^{g+1} m_j^2>0$$
	and hence $c(0)=0$. This proves the theorem.
	 \end{proof}
We can now prove Theorem \ref{lowerbound} as an application of Theorem \ref{generallowerbound}.
\begin{proof}[Proof of Theorem \ref{lowerbound}]
	If we choose $r=2$, $m=(1,1)$, $\alpha=\frac{\omega}{2g-2}$ and 
	$$\overline{\mathcal{M}}=\hat{\Delta}^{\alpha}_{11}+\hat{\Delta}^{\alpha}_{22}+2g\hat{\Delta}^{\alpha}_{12},$$
	we obtain by Theorems \ref{mainthm} and \ref{generallowerbound}
	$$0\ge \langle (f_{m,\alpha}^*\hat{\mathcal{L}}),\overline{\mathcal{M}}^2\rangle=\frac{(-4g^2)(2g+1)}{g-1}\hat{\omega}^2+4g^2\varphi(X).$$
	This is exactly the first bound in the theorem.
	For the two bounds for $g=3$ and $g=4$, we choose $r=4$, $m=(1,1,1,1)$, $\alpha=\frac{\omega}{2g-2}$ and
	$$\overline{\mathcal{M}}=\hat{\Delta}^{\alpha}_{12}-\hat{\Delta}^{\alpha}_{23}+\hat{\Delta}^{\alpha}_{34}-\hat{\Delta}^{\alpha}_{14}.$$
	Again, we obtain by applying Theorems \ref{mainthm} and \ref{generallowerbound} 
	$$0\ge \langle (f_{m,\alpha}^*\hat{\mathcal{L}})^3,\overline{\mathcal{M}}^2\rangle=-4\frac{15g^3-14g^2-19g-6}{g-1}\hat{\omega}^2+8(3g^2-g-6)\varphi(X)$$
	We obtained this expression by evaluating the formula in Theorem \ref{mainthm} with the help of a computer and also by a  long computation by hand. The bounds in the theorem follow by putting in $g=3$ respectively $g=4$.
	 \end{proof}
One may ask, whether in general the choice
$$\overline{\mathcal{M}}=\sum_{j=1}^{r-1} (-1)^j m_j m_{j+1}\hat{\Delta}^{\alpha}_{j,j+1}+(-1)^rm_1m_r\hat{\Delta}^{\alpha}_{1r}$$
for $r$ even leads to good bounds for $\hat{\omega}^2$. But this turns out to be not the case. For $g\ge 5$, the bound $\hat{\omega}^2\ge\frac{g-1}{2g+1}\varphi(X)$ is always better.

\end{document}